\documentclass[12pt, oneside,psamsfonts]{amsart}

\usepackage{amssymb,amsfonts, verbatim}
\usepackage[all,arc]{xy}
\usepackage{enumerate}
\usepackage{mathrsfs}
\usepackage{hyperref}
\usepackage[all]{xy}
\usepackage{amsthm}
\usepackage[T1]{fontenc}
\usepackage{comment} 

\usepackage{cooltooltips}

\usepackage{amssymb, amsmath}
\usepackage{tikz-cd}
\usepackage{MnSymbol}
\usepackage{stmaryrd}

\usepackage{comment}
\usepackage{geometry}

\parskip \baselineskip
\usepackage{graphicx,color}
\hypersetup{
}
\usepackage{bm}

\usepackage{lipsum}

\makeatletter
\newcommand\ackname{Acknowledgements}
\if@titlepage
  \newenvironment{acknowledgements}{%
      \titlepage
      \null\vfil
      \@beginparpenalty\@lowpenalty
      \begin{center}%
        \bfseries \ackname
        \@endparpenalty\@M
      \end{center}}%
     {\par\vfil\null\endtitlepage}
\else
  
\fi
\makeatother

\newtheorem{thm}{Theorem}[section]
\newtheorem{cor}[thm]{Corollary}
\newtheorem{prop}[thm]{Proposition}
\newtheorem{lemma}[thm]{Lemma}

\theoremstyle{definition}
\newtheorem{deflem}[thm]{Definition-Lemma}
\newtheorem{definition}[thm]{Definition}

\newtheorem{exmp}[thm]{Example}
\newtheorem{exmps}[thm]{Examples}

\newtheorem{corollary}[thm]{Corollary}
\usepackage[T3,OT2,T1]{fontenc}
\theoremstyle{remark}
\newtheorem{remark}[thm]{Remark}

\DeclareMathOperator{\Spec}{Spec}
\DeclareMathOperator{\Specm}{Specm}
\DeclareMathOperator{\sep}{sep}
\DeclareMathOperator{\Hom}{Hom}

\DeclareMathOperator{\Der}{Der}

\newcommand{\nn}[1]{ \llbracket 1,  #1 \rrbracket}
\newcommand{\on}[1]{ \llbracket 0,  #1 \rrbracket}
\newcommand{\xn}{\underline{x}_n}

\DeclareOldFontCommand{\bf}{\normalfont\bfseries}{\mathbf}
\theoremstyle{theorem}
\newcommand{\introthmname}{}
\newtheorem{introthminn}{\introthmname}
\newenvironment{introthm}[1]
  {\renewcommand{\introthmname}{#1}\begin{introthminn}}
  {\end{introthminn}}

\newtheorem{introthminn1}{\introthmname}
\newenvironment{introthm1}[1]
  {\renewcommand{\introthmname}{#1}\begin{introthminn1}}
  {\end{introthminn1}}
  
\newtheorem*{thm*}{Theorem}
\makeatletter
\let\c@equation\c@thm
\makeatother
\numberwithin{equation}{section}

\title{A characterization of ramification groups via Taylor morphism}

\author{Sophie Marques and Luigi Pagano}

\begin{document}

\begin{abstract}
In this paper, we present a functorial method to define ramification groups, identifying them as inertia groups of an induced action on composite jet algebras. This framework lays the foundation for defining higher ramification groups for actions involving group schemes. To achieve this, we introduce Taylor maps within the category of commutative unitary rings at prime ideals of an \(R\)-algebra and compute their kernels for algebras of finite type over a field with separably generated residue fields. \\



\\
\end{abstract}
\maketitle \footnote{The first author is an associate of NITheCS (National Institute for Theoretical and Computational Sciences), South Africa.} 

\tableofcontents
\section*{Introduction}\label{d1} 
In algebra, the exploration of field and ring extensions is central. This pursuit finds its roots in various mathematical inquiries, such as the study of Diophantine equations over integers, exemplified by the famous equation \( x^n + y^n = z^n \) for \( n \in \mathbb{N} \), which led to Fermat's last theorem. When considering the case \( n = 2 \), a comprehensive solution emerges by delving into Gaussian integers. Here, the arithmetic properties of this extension closely resemble those of the integers \( \mathbb{Z} \). However, the scenario becomes more intricate for broader cyclotomic extensions \( \mathbb{Q}(\zeta) \), where the ring of integers remains \( \mathbb{Z}[\zeta] \), yet the uniqueness of prime factorization may be forfeited. Various approaches have been proposed, including focusing on prime ideals instead of elements. Crucially, the preservation of essential arithmetic properties hinges upon understanding ramification data. In algebraic number theory, while examining the decomposition of an ideal generated by a prime number in the ring of integers of a Galois number field, it's been established that the common multiplicity in which each prime ideal in the decomposition appears precisely to be the cardinality of the inertia group at that prime, also known as the ramification group of index zero. This theorem extends to both Galois extensions of number fields and Galois function field extensions. A prime ideal which has a ramification index greater than one is said to be ramified. 

The importance of the discriminant of a polynomial or, more generally, of a Galois number field as a crucial invariant of these extensions is undeniable, the discriminant also tells us which primes are ramified. With a finer analysis of the trace map, we can define the different ideal, which is divisible precisely by the ramified primes in ring of integer of the number field in question. The differential exponent of a local field extension such that the residual extension is separable can be computed from the orders of the higher ramification groups for Galois extensions (see \cite[p. 115]{Weiss}). By forming a group filtration, the higher ramification groups also provide valuable insights into the structure of the Galois group under examination.

Historically, the meticulous proofs of Kronecker-Weber by Hilbert, predating class field theory, heavily relied on the utilization of higher ramification groups (see \cite{Conrad2}). In local class field theory, the Hasse-Arf theorem asserts that the increments of the filtration of the higher ramification groups consist entirely of rational integers (see \cite{Hasse}, \cite{Hasse2}, \cite{Cahit}). The definition of the Artin conductor and the Swan character also entails expressions involving higher ramification (see \cite{Serre}).

In \cite{Barnard}, ramification was extended to actions of groups on rings. Ramification is closely intertwined with the preservation, or lack thereof, of arithmetic properties such as Cohen-Macaulayness in the ring of invariants under a given action (see \cite{Lorenz}, \cite{Ben}).

Ramification has also been extended to actions involving group schemes on schemes, although only inertia groups were formally defined. Indeed, in \cite[III, \S, $n^\circ$2]{Gabriel}, a functorial approach to defining inertia groups was introduced. This method extends the definition of inertia beyond just the residue field to include $T$-points, allowing for a more nuanced understanding of inertia groups, which becomes particularly crucial in higher dimensions. Applications to slice theorem, Euler characteristics, and the class-invariant homomorphism can be found in \cite{Boas} and \cite{tame}. Ramification has even been extended and applied to stacks, as discussed in \cite{AOV}, collectively demonstrating the central role of ramification in mathematics. In this article, through jet algebras, we explore an alternative approach to defining higher ramification, with the aspiration of establishing a functorial method for defining them for $T$-points and actions of group schemes on schemes, thereby granting us access to a broader theoretical framework for future investigations.

More precisely, we study the ramification groups induced by a group action on a finite-type algebra \(B\) over a field \(k\). We do so by the means of the jet functor, which allows us to construct a new object that, broadly speaking, manifests properties concerning the behavior of the elements of $B$ in the infinitesimal neighbourhood of the points of its spectrum.
Via this feature, we identify the ramification groups of the original action with ramification groups of lower order of the induced action via this construction.
In Definition \ref{Taylor}, for any \(R\)-algebra \(B\), we define Taylor morphisms $T_{\mathfrak{p}, s}$ of order \(s\) at a prime ideal \(\mathfrak{p}\) of \(B\). 
This morphism is analogous to the Taylor map introduced in a calculus course. This algebraic definition allows us to extend this concept to an algebraic geometry setting, broadening its applicability.  Our first main result computes the kernel of these Taylor morphisms. Specifically,

\begin{introthm}{Theorem}
Let \(k\) be a field, \(s\in \mathbb{N}\setminus \{0 \}\), \(B\) be a \(k\)-algebra of finite type, and \(\mathfrak{p}\) be a prime ideal of \(B\). We assume that its residue field at $\mathfrak{p}$, $k(\mathfrak{p})$, is separably generated over \(k\), and $\operatorname{char}(k)$ is either $0$ or greater than $s$. Then $\operatorname{ker}(T_{\mathfrak{p}, s})=\mathfrak{p}_\mathfrak{p}^{s+1}$.
\end{introthm}

Our second main theorem, a consequence of Theorem 1, proves that the ramification groups for an action of a group on an \(R\)-algebra can be defined as the inertia group of an induced action on certain jet algebras.

\begin{introthm}{Theorem}
Let \(G\) be a group acting on \(B\) a \(k\)-algebra of finite type, \(\mathfrak{p}\) be a prime ideal of \(B\), and \(s\in \mathbb{N}\setminus \{0 \}\). We suppose that its residue field at $\mathfrak{p}$, $k(\mathfrak{p})$, is separably generated over \(k\), and $\operatorname{char}(k)$ is either $0$ or greater than $s$. We have that \(\mathfrak{p}\mathcal{L}_1(B)\) is a prime ideal of \(\mathcal{L}_1(B)\), and the ramification group of order \(s\) at \(\mathfrak{p}\) is given by 
\[G_{s}(\mathfrak{p}) = G_{s-1} (\mathfrak{p}\mathcal{L}_1(B)) =\cdots= G_{0} (\mathfrak{p}\mathcal{L}_1^{s}(B))\]
where 
\begin{itemize}
\item \( G_{s}(\mathfrak{p})\) is the higher-ramification group of order $s$ as defined in Definition \ref{ramification}. 
\item \(\mathcal{L}_1(B)\) is the first jet algebra of \(B\) as defined in Definition \ref{defft};
\item \(\mathcal{L}_1^{t}\) denotes \(\mathcal{L}_1\) composed with itself \(t\) times where $t \in \mathbb{N}\setminus \{0\}$. \end{itemize}
\end{introthm}

As an example, we consider a field \( k \) of characteristic \( 3 \), and the automorphism \( \sigma \) of order \( 3 \) on \( k[x,y] \) defined by \( \sigma(x) = x + y^2 \) and \( \sigma(y) = y \). Here, $G= \langle \sigma \rangle$ and $B = k[x,y]$. At the maximal ideal \( (x,y) \) of \( k[x,y] \), we observe that the inertia group and first ramification groups of this action are both equal to the whole group generated by \( \sigma \), while the higher ramification groups are trivial, indicating a wildly ramified action. 

The jet algebra of the polynomial ring in two variables \( k[x,y] \) can be identified to the polynomial ring in four variables \( k[x_0, y_0, x_1, y_1] \). Moreover, \( \sigma \) induces an order \( 3 \) automorphism \( \mathcal{L}_1(\sigma) \) such that \( \mathcal{L}_1(\sigma)(x_0) = x_0 + y_0^2 \), \( \mathcal{L}_1(\sigma)(x_1) = x_1 + 2 y_0 y_1 \), and \( \mathcal{L}_1(\sigma)(y_i) = y_i \) for \( i \in \{1,2\} \). 

At the prime ideal \( (x_0, y_0) \) of \( k[x_0, y_0, x_1, y_1] \), the inertia group of this induced action remains the whole group generated by \( \sigma \) while the higher ramification group are all trivial, indicating a tamely ramified extension. In particular, the \( n^{th} \) ramification group of \( \sigma \)'s action corresponds to the \( (n-1)^{th} \) ramification group of the induced action \( \mathcal{L}_1(\sigma) \), for all \( n \geq 1 \).

The jet algebra of the polynomial ring in four variable \( k[x_0, y_0, x_1, y_1] \) identifies to the polynomial ring in eight variables \( k[x_{0,0}, y_{0,0}, x_{1,0}, y_{1,0}, x_{0,1}, y_{0,1}, x_{1,1}, y_{1,1}] \), where \( \sigma \) induces an order \( 3 \) automorphism \( \mathcal{L}_1^2(\sigma) \) such that \( \mathcal{L}_1^2(\sigma)(x_{0,0}) = x_{0,0} + y_{0,0}^2 \), \( \mathcal{L}_1^2(\sigma)(x_{1,1}) = x_{1,1} + 2(y_{0,1}y_{1,0} + y_{0,0}y_{1,1}) \), \( \mathcal{L}_1^2(\sigma)(x_{i,j}) = x_{i,j} + 2 y_{0,0} y_{i,j} \) for \( (i,j) \in \{(1,0), (0,1)\} \), and \( \mathcal{L}_1^2(\sigma)(y_{i,j}) = y_{i,j} \) for \( (i,j) \in \{(0,0), (1,0), (0,1), (1,1)\} \). 

At the prime ideal \( (x_{0,0}, y_{0,0}) \) of \( k[x_{0,0}, y_{0,0}, x_{1,0}, y_{1,0}, x_{0,1}, y_{0,1}, x_{1,1}, y_{1,1}] \), this extended action is now unramified, and the \( n^{th} \) ramification group of \( \sigma \)'s action corresponds to the \( (n-2)^{th} \) ramification group of \( \mathcal{L}_1^2(\sigma) \), for all \( n \geq 2 \). 

This example illustrates our approach as a "desingularization" of a wildly ramified action. As we progress to more complex groups and actions, each step in the jet algebra construction enriches our understanding by providing access to a finer level of ramification at each prime ideal. After a finite number of steps, we can compute any ramification group as the inertia group of a composite of jet algebras, as demonstrated earlier. Given the complexity of higher ramification, the assumptions of our theorem allow us to access and explore numerous cases. Additional examples can be found in Example \ref{ex1}, while counterexamples, where the assumptions of the theorem are not met, can be found in Example \ref{ex0}.

Our initial section lays the foundation by presenting background for our subsequent discussions. We compute the space of derivations for different field extensions, as outlined in Corollary \ref{tr} and Lemma \ref{ins}. Additionally, considering a finite separable field extension of a rational field in several variables and its integral closure, we establish in Proposition \ref{seplemma} that the separable maximal spectrum of this integral closure is a dense subset of its spectrum. In Theorem \ref{fs}, we characterize the formal smoothness of the residue field of a prime $\mathfrak{p}$ in a polynomial ring in terms of the density of the separable maximal spectrum of $\mathfrak{p}$ in the set of prime ideals containing $\mathfrak{p}$. From this point on, we assume the formal smoothness of the residue field of the prime considered, which is equivalent to assuming that this residue field is separably generated. Under this assumption and extra assumptions on the characteristic, we characterize when an element belongs to a certain power of a prime in terms of lower power conditions involving a specific set of derivatives, as shown in Lemma \ref{ss-1}. Additionally, we prove that if a prime ideal is the intersection of certain ideals, then locally the power of this prime ideal is the intersection of the powers of those ideals, as discussed in Corollary \ref{int}. As a consequence of these findings, we prove that the ramification group of a certain order at a given prime ideal is the intersection of the ramification groups of the same order at the maximal ideals containing it, a result presented in Corollary \ref{specmsep}.

The second section aims to establish fundamental concepts and notations surrounding jet algebras. Commencing with a recap of the definition of the jet functor (see Definition \ref{jet}), we proceed to revisit the construction of the jet algebras within the finite type case, defining two natural transformations and different constructs that will remain central to our paper (see Definition-Lemma \ref{defft}, Definition-Lemma \ref{defftt}, and Definition-Lemma \ref{dtt}). Subsequently, we elucidate the general construction of jet algebra and introduce the jet algebra functor (see Definition \ref{lm}). This section concludes by establishing an isomorphism between the jet algebra's derivation space and the symmetric algebra of the module of relative differential forms (see Definition-Lemma \ref{sym}).

In the third section, we define and explore the properties of (localized) thickened fibers (see Definition-Lemma \ref{thickfib}). Lemma \ref{Nsp} elucidates our rationale for considering these fibers to describe ramification groups in terms of jet algebras. Finally, we provide a characterization of the fiber of the first jet algebra at the residue field of a prime ideal (see Lemma \ref{irreduc}).

In the final section, we define morphisms of $R$-algebras that behave analogously to the well-known Taylor maps when acting on algebras of finite type (see Definition \ref{Taylor}) as mentioned above. Subsequently, we prove one of our main results: Theorem 1 and Theorem 2 enable us to express ramification groups of an $R$-algebra action in terms of the inertia groups of the induced action on a composition of jet algebras. Such a characterization holds significant value for the definition of ramification groups, especially within the context of algebraic geometry, when considering, for instance, group scheme actions.

\newpage
\section{Index of notation}
In this paper, $R$ denotes a commutative unitary ring and $n,m$ is a natural number.

\begin{center}
  \begin{tabular}{l p{11cm}}
    $\mathbb{N}$ & The natural numbers including $0$; \\
    $\llbracket m,\! n \rrbracket$ & $\{ m, m+1, \ldots, n \}$ with $m < n$; \\
    $R[X]$ & The polynomial ring with indeterminates corresponding to the coordinates of a tuple $X$; \\
    $\underline{x}_{n}$ & $(x_i)_{i\in \llbracket 0,n \rrbracket}$, where $n\in \mathbb{N}\setminus \{ 0 \} $. We will utilize this notation for other symbols in place of $x$; \\
        $\underline{x}_{n, m}^1$ & $(x_{i,j})_{(i,j)\in \nn{n}\times \nn{m}}$, where $n, m\in \mathbb{N} \setminus \{ 0 \}$. We will utilize this notation for other symbols in place of $x$; \\
    $\underline{x}_{n, m}$ & $(x_{i,j})_{(i,j)\in \nn{n}\times \on{m}}$, where $n\in \mathbb{N} \setminus \{ 0 \}$ and $m \in \mathbb{N}$. We will utilize this notation for other symbols in place of $x$; \\
    $[a]_I$ & The class of an element $a$ of some $R[X]$ in $R[X]/I$ where $X$ is some tuple and $I$ is an ideal of $R[X]$; \\
    $[ c ]_{n,C}$ & The class of an element $c\in C[t]$ in $C[t]/(t^{n+1})$, where $C$ is an $R$-algebra and $t$ is an indeterminate. We denote simply $[t]_{n}$ when $C$ is clear from the context; \\
    $\phi_\ast$ & The map $\phi \circ -: \operatorname{Hom}_{ \mathbf{Alg}_R} ( D, B) \rightarrow \operatorname{Hom}_{ \mathbf{Alg}_R} ( D, C)$ sending a map $\psi$ to $\phi \circ \psi$ where $\phi: B \rightarrow C$ is an $R$-algebra morphism; \\
    $\phi^\ast$ & The map $- \circ \phi: \operatorname{Hom}_{ \mathbf{Alg}_R} ( C, D) \rightarrow \operatorname{Hom}_{ \mathbf{Alg}_R} ( B, D)$ sending a map $\psi$ to $\psi \circ \phi$ where $\phi: B \rightarrow C$ is an $R$-algebra morphism; \\
    $\rho_{t, m}(\varphi)$ & The $R$-algebra morphism $\rho_{t,m}(\varphi) :  C [t]/ (t^{m+1}) \rightarrow D [t]/ (t^{m+1})$ sending $\sum_{j=0}^m c_{j} [t]_m^j$ to $\sum_{j=0}^m \varphi (c_{j}) [t]_m^j$ where $\varphi: C \rightarrow D$ is an $R$-algebra morphism, $t$ is an indeterminate, and $m \in \mathbb{N}$; \\
    ${\tau}^C_{m, s} $ & The $(m,s)$-truncation maps ${\tau}^C_{ m, s} \colon  C[t]/(t^{s+1})\rightarrow C[t]/(t^{m+1})$ sending $\sum_{j=0}^s c_{j} [t]_s^j$ to $\sum_{j=0}^m c_{j} [t]_m^j$ where $C$ is an $R$-algebra and  $s, m \in \mathbb{N}$ with $s\geq m$; \\
    $\underline{D}$ & $\operatorname{Hom}_R(D ,-)$ where $D$ is an $R$-algebra;\\
$N_s(\mathfrak{p})$ & $B_\mathfrak{p} / \mathfrak{p}^s B_\mathfrak{p} $ where $\mathfrak{p}$ is a prime ideal of a $R$-algebra $B$;\\
$k(\mathfrak{p})$ & $N_1(\mathfrak{p})$ where $\mathfrak{p}$ is a prime ideal of a $R$-algebra $B$;\\
$\zeta_{\mathfrak{p}, s}$ &\text{The canonical map from $B$ to $N_s(\mathfrak{p})$}.  \\
$\operatorname{Specm} ( B)$ & The maximal spectrum of an $R$-algebra $B$;\\
$\operatorname{Specm}_{J} ( B)$ & $\{ \mathfrak{m} \in  \operatorname{Specm}(B) | J \subseteq \mathfrak{m}\}= V(J) \cap  \operatorname{Specm} ( B)$ where $B$ is an $R$-algebra and $J$ is an ideal of $B$;\\
$\operatorname{Specm}_{ \operatorname{sep}/k} ( B)$  & $ \{ \mathfrak{m} \in  \operatorname{Specm}(B) | k(\mathfrak{m})/ k \text{ is separable } \}$ where $B$ is an $R$-algebra;\\
    $\mathbf{Alg}_R$ & The category of $R$-algebras whose morphisms are $R$-algebra morphisms;\\
$G_s(\mathfrak{p})$& The ramification of order $s$ at a prime ideal $\mathfrak{p}$ of an action of a group $G$ by an $R$-algebra $B$, where $s\in \mathbb{N}$. This is the set of elements of $G$ that stabilize $\mathfrak{p}$ setwise and also act as the identity on  $N_{s+1}(\mathfrak{p})$.
  \end{tabular}
\end{center}

    \newpage
\section{Preliminary background} 
Throughout this section, $R$ is a commutative ring with unity. 

\subsection{Results on derivations} 
We provide a brief recapitulation and proof of several results pertaining to derivations, which we will use in the subsequent discussions for their relationship with the first jet functor (see Definition \ref{lm}).

\begin{definition}
Given a ring $R$, a $R$-algebra $A$ and a $A$-module $M$, we say that a map $d\colon A \rightarrow M$ is a derivation with respect $R$ if it satisfy the following properties:
\begin{itemize}
\item $d(1)=0$;
\item For all $a,b\in A$ we have $d(ab)=a d(b) + b d(a)$, also called ``Leibnitz rule''.
\end{itemize}
We denote by $\operatorname{Der}_R(A,M)$ the set of such derivations.
\end{definition}
The set $\operatorname{Der}_R(A,M)$ is canonically endowed with a structure of $A$-module, since for $d_1,d_2\in \operatorname{Der}_R(A,M)$ and $a\in A$ we have that $d_1+d_2\in \operatorname{Der}_R(A,M)$ and $ad_1\in \operatorname{Der}_R(A,M)$.

\begin{lemma} \label{loc} 
Let \( R \) be a ring, \( A \) an \( R \)-algebra, and \( B \) an \( A \)-algebra with the structural morphism \( \varphi \colon A \rightarrow B \). Consider a derivation \( d \in \operatorname{Der}_R(A, B) \), and a multiplicative set \( S \subseteq A \) such that \( \varphi(S) \subseteq B^\times \). There exists a unique derivation in \( \operatorname{Der}_R(S^{-1} A, B) \) that extends \( d \). Specifically, for all \( a \in A \) and \( b \in S \), it is defined as:
\[ d_{S} \left( \frac{a}{b} \right) = \frac{d(a)}{\varphi(b)} - \frac{ad(b)}{\varphi(b)^2}. \]
When \( S = \{ a^n \mid n \in \mathbb{N} \} \) for some \( a \in A \), we denote \( d_{S} \) as \( d_a \) and when \( S = A \setminus \mathfrak{p}\) for some $\mathfrak{p}$ prime ideal of $A$, we denote \( d_{S} \) as \( d_{\mathfrak{p}} \). We also denote by \( d_{S} \) the canonical derivation in \( \operatorname{Der}_R(S^{-1} A, T^{-1}B) \) induced by \( d \in \operatorname{Der}_R(A, B) \), provided \( d(S) \subseteq T \).
\end{lemma} 


\begin{proof} 
We begin by proving that $d_{S}$ is indeed a derivation which extends $d$.
By definition, $d$ is a morphism of $R$-modules, as for all $a\in A$, $b\in S$, $r\in R$, $$\quad d_{S}\left( \frac{ra}{b} \right) = \frac{rd(a)}{\varphi(b)} - \frac{rad(b)}{\varphi(b)^2} = r\cdot d_{S}\left(\frac{a}b \right), $$
and
$$ d_{S}\left( \frac{a+c}{b} \right) = \frac{d(a+c)}{\varphi(b)} - \frac{(a+c)d(b)}{\varphi(b)^2}= \frac{d(a)}{\varphi(b)} -\frac{ad(b)}{\varphi(b)^2}+ \frac{d(c)}{\varphi(b)} -\frac{cd(b)}{\varphi(b)^2} = d_{S}\left(\frac{a}b \right) + d_{S}\left(\frac{c}b \right) . $$
Moreover $d_{S}(1/1)= d_{S}(1)/1- d_{S}(1)/1=0$ and, for all $a_1,a_2\in A$, $b_1,b_2\in S$, we have
\begin{align*}
 d_{S}\left(\frac{a_1a_2}{b_1b_2} \right) & =  \frac{d(a_1a_2)}{\varphi(b_1b_2)} -\frac{\varphi(a_1a_2)d(b_1b_2)}{\varphi(b_1b_2)^2} 
=\frac{\varphi(a_2)}{\varphi(b_2)} d_{S}\left(\frac{a_1}{b_1} \right) +\frac{\varphi(a_1)}{\varphi(b_1)} d_{S}\left(\frac{a_2}{b_2} \right) ,
\end{align*}
Thus proving that a derivation which extends $d$ exists.

Let $d_{0} \in \operatorname{Der}_R(S^{-1} A, B)$ be a derivation extending $d$, we will prove that $d_0=d_{S}$. For any $b\in B$, we have
$$ 0=d_{0}(1) = d_0\left(\frac{b}b\right) = \varphi(b) d_0\left(\frac{1}b\right) +\frac1{\varphi(b)} d_0\left(\frac{b}{1}\right) . $$
Since $\varphi(B)$ is invertible in $B$ and since $d_0(b/1)=d(b)$, we infer that 
$$ d_0\left(\frac1b\right) = - \frac{d(b)}{\varphi(b)^2}  . $$
Hence, for any $a \in A$, $b\in S$, the following holds:
$$ d_0\left(\frac{a}b\right) = \frac1{\varphi(b)} d(a) - \frac{a d(b)}{\varphi(b)^2} = d_{S}\left(\frac{a}b\right) , $$
thus also the uniqueness is proved. 
\end{proof}


This corollary is a direct consequence of the previous lemma and \cite[Example 1.6]{Qing}.
\begin{corollary} \label{tr} 
Given $k \subseteq F \subseteq L$ fields, with $L/F$ finite separable, $T$ 
a set of algebraically independent indeterminates, $C$ a $L$-vector space, $d$ in $Der_k(F, C)$, and a function $\phi\colon T \rightarrow C$, there exists a unique derivation $d_{L(T),\phi}\in \Der_k ( L(T), C)$ extending $d$ such that for all $t\in T$, $d_{L(T),\phi}(t)= \phi(t)$. 
In particular, $\operatorname{dim}_k (Der_k ( L, C))= |T| = \operatorname{tr}(L/k)$. 
\end{corollary} 


\begin{lemma} \label{ins}
Let $k$ be a field, $F$ be a finitely generated field extension over $k$ where $\operatorname{char}(k)=p>0$, and $E= kF^p$ be the compositum of $k$ and $F^p$ in $F$. We have 
\begin{itemize} 
\item $[F:E]=p^e$, where $e\in \mathbb{N}$, 
\item $F=E(x_1 , \ldots , x_e)$, where $x_i^p =t_i\in E$ for all $i \in \{ 1, \ldots, e\}$, 
\item $[E(x_1 , \ldots, x_{i} ) : E(x_1, \ldots , x_{i-1})]=p$, for all $i \in \{ 2, \ldots, e\}$, and 
\item for any $C$ $F$-module, 
$$\begin{array}{lrll} \Phi_C : & \operatorname{Der}_k(F, C) & \rightarrow & C^e \\ 
& d & \mapsto & (d(x_i))_{i \in \{ 1, \ldots, e\}} \end{array}$$
is an $R$-linear isomorphism. 
Equivalently, $\Omega^1_{F/k} \simeq F^e$. 
\end{itemize}
\end{lemma} 

\begin{proof}
By definition of $E$, every element $z\in F$ satisfies $z^p \in E$. Therefore, $F$ is algebraic over $E$, and since it is finitely generated, $[F:E]=p^e$, for some $e \in \mathbb{N}$. Using induction on $e$, one can prove that $F=E(x_1 , \ldots , x_e)$ where $x_i^p =t_i\in E$ for all $i \in \{ 1, \ldots, e\}$, and $[E(x_1 , \ldots, x_{i} ) : E(x_1, \ldots , x_{i-1})]=p$, for all $i \in \{ 2, \ldots, e\}$. Now, let $C$ be an $F$-module and $d \in \operatorname{Der}_k(F, C)$, we have $d|_{F^p} \equiv 0$. Indeed, for all $x\in F^p$, there exists $a \in F$ such that $x=a^p$. Thus $d(x) = p \cdot a^{p-1}  \cdot d(a) =0$. Therefore, since $d|_k\equiv 0$, $d_E \equiv 0$, the map 
$$\begin{array}{lrll} \Phi : & \operatorname{Der}_k(F, C) & \rightarrow & C^e \\ 
& d & \mapsto & (d(x_i))_{i \in \{ 1, \ldots, e\}} \end{array}$$
is an $F$-linear isomorphism. 
\end{proof}


\subsection{Properties of the maximal separable spectrum}
We begin this section by introducing notation related to the spectrum of a ring.

\begin{definition}
Let $k$ be a field and $B$ be a $k$-algebra.
\begin{enumerate}
    \item We define the \textsf{maximal spectrum of $B$}, denoted by $\operatorname{Specm}(B)$, as the set consisting of all the maximal ideals of $B$.
    \item For an ideal $J$ of $B$, we define the \textsf{maximal spectrum of $B$ at $J$} as the set
    \[
    \operatorname{Specm}_{J}(B) = \{\mathfrak{m} \in \operatorname{Specm}(B) \mid J \subseteq \mathfrak{m}\} = V(J) \cap \operatorname{Specm}(B).
    \]
    \item We define the \textsf{maximal separable spectrum} as
    \[
    \operatorname{Specm}_{\operatorname{sep}/k}(B) = \{\mathfrak{m} \in \operatorname{Specm}(B) \mid k(\mathfrak{m})/k \text{ is separable}\}.
    \]
\end{enumerate}
\end{definition}

We prove that the maximal separable spectrum of a polynomial ring over a field is dense in the corresponding spectrum.

\begin{lemma}
\label{lem:densityfreealg}
Let $k$ be a field. Then $\Specm_{\sep /k}(k[t_1,\dots,t_n])$ is dense in $\Spec(k[t_1,\dots, t_n])$ with respect to the Zariski topology.
\end{lemma}
\begin{proof}
It is enough to prove that there exists $\mathfrak{m}\in \Specm_{\sep /k}(k[t_1,\dots, t_n]_h)$, for any $h\in k[t_1,\dots, t_n]\backslash \{0\}$. 
To prove this statement, we proceed by induction on $n$.

If $n=1$, the result is true because $k^{\sep}$ is an infinite field and $h$ has a finite number of roots.

Now assume that the result holds for all polynomials in $k[t_1,\dots,t_{n-1}]$. We will prove the result for all polynomials in $k[t_1,\dots,t_{n}]$. Let $h \in k[t_1,\dots,t_{n}] \backslash \{ 0 \}$. We write
$$ h=\sum_{j=0}^m h_j x_n^j \, ,$$
where $h_j\in k[t_1,\dots, t_{n-1}]$ for all $0\leq j\leq m $. Since $h\neq 0$, there exists $j \in \{ 0 , \dots  , m\}$ such that $h_j\neq 0$. By the inductive hypothesis, there exist $\alpha_1,\dots,\alpha_{n-1}\in {k^{\sep}}$ such that $h_j(\alpha_1,\dots,\alpha_{n-1})\neq 0$. In particular, the polynomial $h(\alpha_1,\dots,\alpha_{n-1}, t)\in k^{\sep}[t]$ is not equal to $0$. Hence, there exists $\alpha_{n}\in k^{\sep}$ such that $h(\alpha_1,\dots, \alpha_n)\neq 0$. 

Consider the morphism
\begin{align*}
\varphi_{\alpha} \colon k[t_1,\dots, t_n] &\rightarrow k^{\sep} \\
t_i & \mapsto \alpha_i \, .
\end{align*}
Let $\mathfrak{m}\colon = \ker(\varphi_{\alpha})$. Then $\mathfrak{m}$ is a maximal ideal with a separable residue field since $k[t_1,\dots, t_n]/\mathfrak{m}$ embeds in $k^{\sep}$. Moreover, $h\notin \mathfrak{m}$ as $h(\alpha_1,\dots,\alpha_n)\neq 0$. This concludes the proof.
\end{proof}

In the forthcoming lemma, we demonstrate the density of the separable maximal spectrum within the spectrum of an integral closure.

\begin{prop} \label{seplemma}
Let $k$ be a field, and $F/k(t_1,\dots, t_n)$ be a finite separable field extension. Let $\mathcal{O}_F$ be the integral closure of $k[t_1,\dots,t_n]$ in $F$. Then $\Specm_{\sep /k}(\mathcal{O}_F)$ is a dense subset of $\Spec(\mathcal{O}_F)$.
\end{prop}

\begin{proof}
Since $F/k(t_1,\dots, t_n)$ is a finite separable extension, by the primitive element theorem, there exists an element $y \in F$ such that $F=k(t_1,\dots, t_n)$. Without loss of generality, we can assume that $y \in  \mathcal{O}_F$.

We know that ${\mathcal{O}_F}$ admits an integral basis $\{\alpha_1, \dots, \alpha_s\}$ over $k[t_1, \dots,t_n]$ (see for instance \cite[Proposition 2.10]{Neukirch}). Moreover, for all $i \in \{ 1, \dots, s\}$, we have that $\alpha_i = \sum_{j=1}^{l_i} d_{i,j} y^j$, where $d_{i,j}=\frac{h_{i,j}}{g_{i,j}}$, for some $h_{i,j} , g_{i,j}\in k[t_1, \dots,t_n]$. We set $g:= \prod_{i=1}^s  \prod_{j=1}^{l_i} g_{i,j}$. Clearly, the inclusion $\iota: k[t_1,\dots, t_n][y] \rightarrow \mathcal{O}_{F}$ induces also an isomorphism of $k[t_1,\dots, t_n]_g$-algebras 
$$\iota_g \colon k[t_1,\dots, t_n]_g[y]\rightarrow \mathcal{O}_{F,g}.$$ Up to replacing $g$ with one of its multiples, we assume that the discriminant of the minimal polynomial of $y$ over $k[t_1, \dots,t_n]$ divides $g$. Since $g\neq 0$, $\Spec(\mathcal{O}_{F,g})$ is a dense open subset of $\Spec(\mathcal{O}_F)$. Therefore, it is enough to prove that $\Specm_{\sep /k}(\mathcal{O}_{F,g})$ is dense in $\Spec(\mathcal{O}_{F,g})$ to achieve our conclusion.

Let $h\in \mathcal{O}_{F,g}\backslash\{0\}$. Noting that $\mathcal{O}_{F,g}$ is the integral closure of $k[t_1,\dots, t_n]_g$ in $F$. We denote by $p(x)$ the minimal polynomial of $h$ over $k[x_1, \dots  , x_n]_g$. Since $p$ is irreducible, $p(0)\neq 0$. Hence, by Lemma \ref{lem:densityfreealg}, we infer that there exists $ \mathfrak{m}\in \Specm_{\sep /k}(k[t_1,\dots, t_n]_{g})$ such that $p(0) \notin  \mathfrak{m}$. For any $\mathfrak{n}\in \Specm(\mathcal{O}_{F,g})$ containing $\mathfrak{m}\mathcal{O}_{F,g}$, then $h \notin \mathfrak{n}$ (otherwise the product $p(0)$ of the conjugates of $h$ would be in $\mathfrak{m}$), and $k(\mathfrak{n})$ is separable over $k$. Indeed, we have that 
$$
\mathcal{O}_{F,g}/\mathfrak{m}\mathcal{O}_{F,g} \cong k[t_1,\dots, t_n]_g[x]/(f) \otimes_{k[t_1,\dots, t_n]_g} k(\mathfrak{m})=k(\mathfrak{m})[x]/(f(c_1,\dots,c_n)(x))  \, ,
$$
where $c_i= t_i +\mathfrak{m}\in k(\mathfrak{m})$ and $f$ is the minimal polynomial of $y$ over $k[t_1, \dots,t_n]$; since we assumed that the discriminant of $f$ divides $g$ and that $g\notin \mathfrak{m}$, we infer that $f(c_1,\dots,c_n,x)$ has distinct roots in a fixed algebraic closure of $k(\mathfrak{m})$, therefore $k(\mathfrak{n})/k(\mathfrak{m})$ is separable, which, in turn, means that $k(\mathfrak{n})/k$ is separable. This concludes the proof.
\end{proof}
\subsection{Formal smoothness}
In this section, we establish various characterization of formal smoothness. 
We start with a lemma which is a direct consequence of \cite[Lemma 10.158.8]{Stack} and the definition of formal smoothness.
\begin{lemma}\label{graded} 
Let $B$ be an algebra over $k$, and $\mathfrak{p}$ be a prime ideal of $B$ such that $k(\mathfrak{p})$ is separably generated over $k$. 
Then the canonical morphism $\pi_{2,1, \mathfrak{p}} : N_2(\mathfrak{p}) \rightarrow k( \mathfrak{p})$ admits a section, and $N_2 (\mathfrak{p}) \simeq k(\mathfrak{p}) \oplus \mathfrak{p}_\mathfrak{p} / \mathfrak{p}^2_\mathfrak{p} $ as an algebra.
\end{lemma} 
In the following lemma, we provide characterizations of formal smoothness that will be useful in our subsequent discussions.
\begin{thm}\label{fs}
Let $A = k[x_1, \dots , x_n]$ be a polynomial ring in $n$ indeterminates, $\mathfrak{p}$ be a prime ideal of $A$, and $\bar{k}$ be an algebraic closure of $k$. The following statements are equivalent:
\begin{enumerate}
    \item $k(\mathfrak{p})$ is formally smooth over $k$.
    \item There exist $r \in \mathbb{N}$ and $k$-derivations $d_i\in \operatorname{Der}_k(A,A)$ and $g_i \in A$ for all $i \in \nn{r}$ such that $\mathfrak{p}A_\mathfrak{p} = \langle \frac{g_1}{1}, \dots , \frac{g_r}{1} \rangle $, and $\det((d_i(g_j))_{(i,j) \in \nn{r}^2 } ) \notin \mathfrak{p}$.
    \item $k(\mathfrak{p})$ is separably generated over $k$.
    \item $k(\mathfrak{p}) \otimes \bar{k}$ is reduced.
    \item $\operatorname{Specm}_{\operatorname{sep}}(A) \cap V(\mathfrak{p})$ is dense in $V(\mathfrak{p})$.
    \item $\mathfrak{p} = \bigcap_{\mathfrak{m} \in  \operatorname{Specm}_{\operatorname{sep}}(A) \cap V(\mathfrak{p})} \mathfrak{m}$.
\end{enumerate}
\end{thm}

\begin{proof}
Let $S = \{ x_1, \dots , x_n \}$, and we write $k[S] := k[x_1, \dots , x_n]$. Let $\mathfrak{p} \in \operatorname{Spec}(k[S])$.

 $(1) \Leftrightarrow (2)$ This equivalence is guaranteed by \cite[Theorem 64]{Matsumura} applied to $A = k[x_1, \dots , x_n]$, $B = A/\mathfrak{p}$ at the prime ideal $(0)$ since $A$ is smooth over $k$.
   
 $(1) \Leftrightarrow (3)$ This follows from \cite[Corollary 2.4.6]{Smooth} or \cite[Theorem 62]{Matsumura} or \cite[Lemma 10.158.8.]{Stack}.
    
 $(1) \Leftrightarrow (4)$ This follows from \cite[(27.F) Lemma 3]{Matsumura}.
    
 $(5) \Leftrightarrow (6)$ This follows from the definition of density in $V(\mathfrak{p})$ with respect to the Zariski Topology.
    
 $(3) \Rightarrow (5)$ Let $R =A/\mathfrak{p}$. As $k(\mathfrak{p})/k$ is separably generated, by \cite[Lemma 10.42.3.]{Stack}, there exists a finite subset $T =\{ t_1, \dots, t_r\} \subseteq k(\mathfrak{p})$ and $\theta \in k(\mathfrak{p})$ such that $k(\mathfrak{p}) / k( T)$ is finite separable and $k(\mathfrak{p}) = k(T)(\theta)$. We denote $\overline{k[T]}$ to be the integral closure of $k[T]$ in $k(\mathfrak{p})$. Without loss of generality, we can assume that $\theta \in \overline{k[T]}$. We recall that $k(\mathfrak{p})$ is the quotient $A_{\mathfrak{p}} / (\mathfrak{p}A_{\mathfrak{p}})$. For all $i \in \{ 1, \dots , r \}$, $t_i= [ \frac{a_i}{b_i}]_{\mathfrak{p}A_{\mathfrak{p}}}$ with $a_i \in A$ and $b_i \in A \setminus \mathfrak{p}$. We set $f=[ \prod_{i=1}^r b_i]_{\mathfrak{p}A_{\mathfrak{p}}}$. We deduce that that $R_{f}$ embeds in $\overline{k[T]}$. By \cite[Proposition 2.10]{Neukirch}, $\overline{k[T]}$ admits an integral basis $\{\alpha_1, \dots, \alpha_s\}$ over $k[T]$.  Moreover, for all $i \in \{ 1, \dots, s\}$, we have that $\alpha_i = \sum_{j=1}^{l_i} d_{i,j} \theta^j$, where $d_{i,j}=\frac{h_{i,j}}{g_{i,j}}$, for some $h_{i,j} , g_{i,j}\in k[T]$. We set $g:= \prod_{i=1}^s  \prod_{j=1}^{l_i} g_{i,j}$. Finally, $R_f \simeq \overline{ k[T]}_g$. Since $\Spec (R_f)$ is not empty and $R$ is irreducible, $\Spec (R_f)$  is dense in $\Spec(R)$. Similarly, $\Spec (\overline{ k[T]}_g)$ is dense in $\Spec (\overline{ k[T]})$. Therefore, $\Specm_{\sep /k} (R)$  is dense in $\Spec (R)$ if and only if $\Specm_{\sep} (\overline{ k[T]})$  is dense in $\Spec (\overline{ k[T]})$. By Lemma \ref{seplemma}, we infer that $\Specm_{\sep /k } (\overline{ k[T]})$ is dense in $\Spec (\overline{ k[T]})$. Thus, $\Specm_{\sep /k } (R)$  is dense in $\Spec (R)$. 
    
 $(5) \Rightarrow (4)$ We prove the implication by proving its contrapositive. Assume that $k(\mathfrak{p})\otimes_k \bar{k}$ is not reduced. Let $C \colon= A/ \mathfrak{p} \otimes_k \bar{k}$. By \cite[Chapter 3. Exercise 5]{AM}, $C$ is not reduced. We prove that for any maximal ideal $\mathfrak{m}\subseteq A$ containing $\mathfrak{p}$, $k(\mathfrak{m})/k$ is not separable. It follows from \cite[Theorem 4.21]{CTP} that $k(\mathfrak{m}) \otimes_k \bar{k} \simeq k(\mathfrak{m}) \otimes_{A/ \mathfrak{p} }C\simeq k(\mathfrak{m}) \otimes_{k }C $. Since $k(\mathfrak{m}) $ is faithfully flat over $k$ and since $\sqrt{0_C}  \neq \{0\}$ by construction, we have $\sqrt{0_C} \otimes_k k(\mathfrak{m})\neq \{ 0\} $. Thus, since $ \sqrt{0_C} \otimes_k k(\mathfrak{m}) \subseteq \sqrt{0_{ k(\mathfrak{m}) \otimes_{k }C}}$, we infer that $k(\mathfrak{m}) \otimes_k \bar{k} $ is not reduced. Hence $k(\mathfrak{m})/k$ is not separable.

\end{proof}
The lemma presented below serves as the principal tool in proving our main Theorem 1. 
\begin{lemma}
\label{ss-1}
Let \( k \) be a field, \( n \in \mathbb{N} \), \( A = k[x_1, \dots , x_n] \), \( \mathfrak{p} \) a prime ideal of \( A \), \( s \geq 2 \), \( f \in A \), and \( I = \{1, \dots , r\} \). Assume \( k(\mathfrak{p}) \) is separably generated over \( k \), and \( \operatorname{char}(k) \) is either \( 0 \) or greater than \( s \). Choose \( g_1, \dots, g_r \in \mathfrak{p} \) and \( d_1, \dots, d_r \in \operatorname{Der}_k(A,A) \) as in Lemma \ref{fs}. Then, there exists \( a \notin \mathfrak{p} \) such that \( \mathfrak{p}_a = \langle \frac{g_i}{1} \rangle_{i \in I} \), and the following equivalence holds: \( f \in \mathfrak{p}_a^s \) if and only if \( f \in \mathfrak{p}_a^{s-1} \) and \( {d_i}_a(f) \in \mathfrak{p}_a^{s-1} \) for all \( i \in I \), where \( {d_i}_a \in \operatorname{Der}_k(A_a,A_a) \) as defined in Lemma \ref{loc} for all \( i \in I \). In particular, \( f \in \mathfrak{p}_\mathfrak{p}^s \) if and only if \( f \in \mathfrak{p}_\mathfrak{p}^{s-1} \) and \( {d_i}_\mathfrak{p}(f) \in \mathfrak{p}_\mathfrak{p}^{s-1} \) for all \( i \in I \), where \( {d_i}_\mathfrak{p} \in \operatorname{Der}_k(A_\mathfrak{p},A_\mathfrak{p}) \) as defined in Lemma \ref{loc} for all \( i \in I \).
\end{lemma}

\begin{proof}
Let \( k \) be a field, \( n \in \mathbb{N} \), \( A = k[x_1, \dots , x_n] \), \( \mathfrak{p} \) a prime ideal, \( s \in \mathbb{N} \), and \( f \in A \).

Clearly, \( f \in \mathfrak{p}^s \) implies \( f \in {\mathfrak{p}}^{s-1} \) and \( d_i(f) \in {\mathfrak{p}}^{s-1} \) for all \( i \in I \).

Now, we prove the converse. Since \( \mathfrak{p}A_\mathfrak{p} = \langle \frac{g_i}{1} \rangle_{i \in I} \) and \( \mathfrak{p} \) is finitely generated, we have \( \mathfrak{p} = \langle g_i, h_j \rangle_{(i,j) \in I\times J} \) for some finite set \( J \) and \( h_j \in A \) for all \( j \in J \).

Since \( \mathfrak{p}A_\mathfrak{p} = \langle \frac{g_i}{1} \rangle_{i \in I} \), there exists \( \delta \in A \) such that for all \( j \in J \), \( \frac{h_j}{1} = \sum_{i \in I} \frac{a_{i,j}}{\delta} \frac{g_i}{1} \), where \( a_{i,j} \in A \) for all \( (i,j) \in I \times J \). We set \( a = \det((d_i({g_j}))_{i,j \in I})\delta \), and hence \( \mathfrak{p}_a = \langle \frac{g_i}{1} \rangle_{i \in I} \).

To prove the result, it suffices to show that for any \( f = \sum_{\underline{t} \in K} a_{\underline{t}}  \prod_{j=1}^{k} \left( \frac{g_{j}}{1}\right)^{t_{j}} \) with \( K = \{ \underline{t}= (t_{1}, \cdots, t_{{k}}) \in \mathbb{N}^k \mid t_{1}+t_{2}+ \cdots+ t_{k}=s\} \), and \( d_r(f) \in \mathfrak{p}^s \) for all \( r \in \{ 1,2, \cdots, k\} \), then \( a_{\underline{t}} \in \mathfrak{p} \) for all \( \underline{t} \in K \). Here, \( a_{\underline{t}} \) and \( t_j \) are allowed to be zero. We prove this by induction on \( s \).

{\bf Base Case \( s = 1 \):}

Suppose \( f = \sum_{i=1}^l a_{i} g_{i} \) with \( l \leq k \) and \( d_r(f) = 0 \) for all \( r \in \{ 1, \cdots, k\} \).

For \( r \in \{ 1,2, \cdots, k\} \), we have:
\[ d_r(f) = \sum_{i=1}^l d_r(a_{i}) \left( \frac{g_{i}}{1}\right)+ \sum_{i=1}^l a_i d_r\left( \frac{g_{i}}{1}\right) \in \mathfrak{p}. \]

Thus, \( \sum_{i=1}^l a_i d_r\left( \frac{g_{i}}{1}\right) \in \mathfrak{p} \) for all \( r \in \{ 1, \cdots, k\} \). From this, we deduce:
\[ \left( \begin{array}{c} d_i\left( \frac{g_{j}}{1}\right) \end{array}\right)_{(i,j)} \left( \begin{array}{c} a_i \end{array}\right)_{i} =\left( \begin{array}{c} u_i \end{array}\right)_{i} \]
where \( u_i \in \mathfrak{p} \) for all \( i \in \{ 1, \cdots, k\} \). Since \( \left( \begin{array}{c} d_i\left( \frac{g_{j}}{1}\right)\end{array}\right)_{(i,j)} \) is invertible, we obtain \( a_i \in \mathfrak{p} \) for all \( i \in \{ 1, \cdots, k\} \). Thus, the base case holds.

{\bf Inductive Step:} Suppose the induction assumption is true for \( s-1 \geq 2 \). Consider \( f = \sum_{\underline{t} \in K} a_{\underline{t}}  \prod_{j=1}^{k} \left( \frac{g_{j}}{1}\right)^{t_{j}} \) with \( K = \{ \underline{t}= (t_{1}, \cdots, t_{{k}}) \in \mathbb{N}^k \mid t_{1}+t_{2}+ \cdots+ t_{k}=s\} \), and \( d_r(f) \in \mathfrak{p}^s \) for all \( r \in \{ 1,2, \cdots, k\} \).

We have:
\[ d_r(f) = \sum_{\underline{t} \in K} \left( d_r (a_{\underline{t}}) \prod_{j=1}^{k} \left( \frac{g_{j}}{1}\right)^{t_{j}} \right) + \sum_{l=1}^{k} \left( \sum_{\underline{t} \in K} t_{l} a_{\underline{t}}   \left( \frac{g_{l}}{1}\right)^{{t_{l}}-1}  \prod_{j\neq l} \left( \frac{g_{j}}{1}\right)^{{t_{j}}}   \right)d_{r} \left( \frac{g_{l}}{1}\right). \]

We deduce that \( \sum_{l=1}^{k} \left( \sum_{\underline{t} \in K} t_{l} a_{\underline{t}}   \left( \frac{g_{l}}{1}\right)^{{t_{l}}-1}  \prod_{j\neq l} \left( \frac{g_{j}}{1}\right)^{{t_{j}}}   \right)d_{r} \left( \frac{g_{l}}{1}\right)  \in \mathfrak{p}^s \).

Thus, we have:
\[ \left( \begin{array}{c} d_i(g_j)\end{array}\right)_{(i,j)} \left( \begin{array}{c}  \sum_{\underline{t} \in K} t_{l} a_{\underline{t}}   \left( \frac{g_{l}}{1}\right)^{{t_{l}}-1}  \prod_{j\neq l} \left( \frac{g_{j}}{1}\right)^{{t_{j}}}    \end{array}\right)_{i} =\left( \begin{array}{c} u_i \end{array}\right)_{i} \]
where \( u_i \in \mathfrak{p}^s \) for all \( i \in \{ 1, \cdots, k\} \).

Since \( \left( \begin{array}{c} d_i(g_j)\end{array}\right)_{(i,j)} \) is invertible, we obtain \( \sum_{\underline{t} \in K} t_{l} a_{\underline{t}}   \left( \frac{g_{l}}{1}\right)^{{t_{l}}-1}  \prod_{j\neq l} \left( \frac{g_{j}}{1}\right)^{{t_{j}}}  \in \mathfrak{p}^s \) for all \( l \in \{ 1, \cdots, k\} \).

Therefore, for each \( l \in \{ 1, \cdots, k\} \), we obtain:
\[ \sum_{\underline{t} \in K} t_{l} a_{\underline{t}}   \left( \frac{g_{l}}{1}\right)^{{t_{l}}-1}  \prod_{j\neq l} \left( \frac{g_{j}}{1}\right)^{{t_{j}}}    \in \mathfrak{p}^{s-1} \]
and
\[ d_r \left( \sum_{\underline{t} \in K} t_{l} a_{\underline{t}}   \left( \frac{g_{l}}{1}\right)^{{t_{l}}-1}  \prod_{j\neq l} \left( \frac{g_{j}}{1}\right)^{{t_{j}}} \right) \in \mathfrak{p}^{s-1} \text{ for all } r \in \{ 1, \cdots, k\}. \]

Applying the induction hypothesis, we obtain \( t_{l} a_{\underline{t}} \in \mathfrak{p} \) for all \( \underline{t} \in K \) and \( l \in \{ 1, \cdots, k\} \).
Fix $\underline{t}\in K$; since $\sum_{l=1}^k t_l = s$, then for some \( l \in \{ 1, \cdots, k\}\), we have that $t_l>0$ and clearly $t_l\leq s$. Given our assumption on the characteristic, this implies that \( a_{\underline{t}} \in \mathfrak{p} \). 
This concludes the induction.
\end{proof}

Based on the previous lemma, we establish that if a prime ideal is the intersection of certain maximal ideals, then locally, under the same assumptions as in the lemma, the power of this prime ideal is the intersection of the powers of those ideals.

\begin{cor} \label{int}
Let \( k \) be a field, \( B \) an algebra of finite type over \( k \), \( \mathfrak{p} \) a prime ideal of \( B \), and \( s \in \mathbb{N} \).
Assume that \( k(\mathfrak{p}) \) is separably generated over \( k \), and that \( \operatorname{char}(k) \) is either \( 0 \) or greater than \( s \).
Then there exists \( a \in A \) such that \( \mathfrak{p}_a^{s} = \bigcap_{\mathfrak{m} \in  \operatorname{Specm}_{\operatorname{sep}}(R) \cap V ( \mathfrak{p})} \mathfrak{m}_a^s \).
\end{cor}

\begin{proof}
By Lemma \ref{fs}, we know that $\mathfrak{p} = \bigcap_{\mathfrak{m} \in \operatorname{Specm}_{\operatorname{sep}}(R) \cap V(\mathfrak{p})} \mathfrak{m}$. Moreover, $\operatorname{Specm}_{\operatorname{sep}}(R)$ ensures that $k(\mathfrak{m})/k$ is separable for all $\mathfrak{m} \in \operatorname{Specm}_{\operatorname{sep}}(R) \cap V(\mathfrak{p})$.

We begin by proving the result for the case where $B = k[x_1, \dots, x_n]$, a polynomial ring using induction on $s$. For $s=1$, the result follows directly from $\mathfrak{p} = \bigcap_{\mathfrak{m} \in \operatorname{Specm}_{\operatorname{sep}}(R) \cap V(\mathfrak{p})} \mathfrak{m}$. By Lemma \ref{fs}, we can choose $r \in \mathbb{N}$, $k$-derivations $d_i \in \operatorname{Der}_k(B, B)$, and $g_i \in B$ for all $i \in \{1, \dots, r\}$ such that $\mathfrak{p}B_\mathfrak{p} = \langle \frac{g_1}{1}, \dots , \frac{g_r}{1} \rangle$, and $\det((d_i(g_j))_{1 \leq i,j \leq r}) \notin \mathfrak{p}$. Since $\mathfrak{p}B_\mathfrak{p} = \langle \frac{g_i}{1} \rangle$, there exists $\delta \in A$ such that for all $j \in \{1, \dots, r\}$, $\frac{h_j}{1} = \sum_{i=1}^r \frac{a_{i,j}}{\delta} \frac{g_i}{1}$, where $a_{i,j} \in B$. We set $a = \det((d_i(g_j))_{1 \leq i,j \leq r})\delta$ as in the proof of Lemma \ref{ss-1}. We assume that $\mathfrak{p}_a^{s-1} = \bigcap_{\mathfrak{m} \in \operatorname{Specm}_{\operatorname{sep}}(R) \cap V(\mathfrak{p})} \mathfrak{m}_a^{s-1}$, for some $s \geq 2$. We will prove that $\mathfrak{p}_a^{s} = \bigcap_{\mathfrak{m} \in \operatorname{Specm}_{\operatorname{sep}}(R) \cap V(\mathfrak{p})} \mathfrak{m}_a^{s}$.

We claim that the following statements are equivalent:
\begin{enumerate}
    \item $f \in \mathfrak{p}_a^{s}$;
    \item $f \in \mathfrak{p}_a^{s-1}$ and ${d_i}_a (f) \in {\mathfrak{p}_a}^{s-1}$ for all $i \in \{1, \dots, r\}$;
    \item $f \in \mathfrak{m}_a^{s-1}$ and ${d_i}_a (f) \in {\mathfrak{m}_a}^{s-1}$ for all $i \in \{1, \dots, r\}$ and for all $\mathfrak{m} \in \operatorname{Specm}_{\operatorname{sep}}(R) \cap V(\mathfrak{p})$;
    \item $f \in \mathfrak{m}_a^{s}$ for all $\mathfrak{m} \in \operatorname{Specm}_{\operatorname{sep}}(R) \cap V(\mathfrak{p})$.
\end{enumerate}
Indeed, by Lemma \ref{ss-1}, we infer that $(1) \Leftrightarrow (2)$ and $(3) \Leftrightarrow (4)$, while by the inductive assumption $(2) \Leftrightarrow (3)$. The desired statement follows from the induction hypothesis and those equivalences.

Now we prove the result in the general case for $B$ as a finitely presented algebra over $k$. Since $B \cong k[x_1, \dots, x_n] / \mathfrak{a}$ for some ideal $\mathfrak{a} \subseteq k[x_1, \dots, x_n]$, we assume without loss of generality that $B = k[x_1, \dots, x_n] / \mathfrak{a}$. By the ideal correspondence, for a prime ideal $\mathfrak{p} \subseteq B$, there exists a prime ideal $\mathfrak{P} \subseteq k[x_1, \dots, x_n]$ containing $\mathfrak{a}$ such that $\mathfrak{p} = \mathfrak{P} / \mathfrak{a}$. Moreover, we have $\operatorname{Specm}_{\operatorname{sep}}(B) \cap V(\mathfrak{p}) = \{\mathfrak{M} / \mathfrak{a} \mid \mathfrak{M} \in \operatorname{Specm}_{\operatorname{sep}}(k[x_1, \dots, x_n]) \cap V(\mathfrak{P})\}$. Now $\mathfrak{p}_a = \mathfrak{P}_a / \mathfrak{a}_a$, hence $\mathfrak{p}_a = \bigcap_{\mathfrak{M} \in \operatorname{Specm}_{\operatorname{sep}}(k[x_1, \dots, x_n]) \cap V(\mathfrak{P})} (\mathfrak{M}_a / \mathfrak{a}_a)$. Therefore,
\[
\mathfrak{p}_a^s = \left(\mathfrak{P}_a / \mathfrak{a}_a\right)^s = \left(\mathfrak{P}_a^s + \mathfrak{a}_a\right) / \mathfrak{a}_a = \frac{\left(\bigcap_{\mathfrak{M} \in \operatorname{Specm}_{\operatorname{sep}}(k[x_1, \dots, x_n]) \cap V(\mathfrak{P})} \mathfrak{M}_a^s\right) + \mathfrak{a}_a}{\mathfrak{a}_a}.
\]
This simplifies to $\mathfrak{p}_a^s = \bigcap_{\mathfrak{m} \in \operatorname{Specm}_{\operatorname{sep}}(B) \cap V(\mathfrak{p})} \mathfrak{m}_a^s$, concluding the proof.
\end{proof}

From the preceding result, we conclude that the ramification group of a given order at a prime ideal $\mathfrak{p}$ is the intersection of the ramification groups of the same order at the maximal ideals $\mathfrak{m}$ containing $\mathfrak{p}$.

\begin{cor}\label{specmsep}
Let $B$ be a finitely generated algebra over a field $k$ equipped with a group action of a group $G$, $\mathfrak{p}$ a prime ideal of $B$, and $s \in \mathbb{N}$. Assume $k(\mathfrak{p})$ is separably generated over $k$, and $\operatorname{char}(k)$ is either $0$ or greater than $s$. Then, the ramification group of order $s$ at $\mathfrak{p}$ is given by
\[
G_s(\mathfrak{p}) = \bigcap_{\mathfrak{m} \in \operatorname{Specm}_{\operatorname{sep}}(B) \cap V(\mathfrak{p})} G_s(\mathfrak{m}).
\]
\end{cor}

\section{The construction of a jet algebra}

In this section, we 
present the notion of jet algebras, which will be involved in our subsequent discussions. We describe some classical properties of such objects starting with the definition of the $m^{th}$ jet functor of an algebra.
\begin{definition}\cite[\S 2]{EM} \label{jet}
Let $m\in \mathbb{N}$ and $B$ be an $R$-algebra. We define the \textsf{$m^{th}$ jet functor of $B$}, denoted $\underline{J_m(B)}$, as a covariant functor from $\mathbf{Alg}_R$ to $\mathbf{Set}$ that assigns to each $R$-algebra $C$ the set $\operatorname{Hom}_{\mathbf{Alg}_R}( B ,C [t]/ (t^{m+1}) )$, and to each morphism of $R$-algebras $\varphi$ the corresponding map $\rho_{t,m}(\varphi)_\ast$ where $\rho_{t,m}(\varphi) :  C [t]/ (t^{m+1}) \rightarrow D [t]/ (t^{m+1})$ is the $R$-algebra morphism sending $\sum_{j=0}^m c_{j} [t]_m^j$ to $\sum_{j=0}^m \varphi (c_{j}) [t]_m^j$ and $ \rho_{t,m}(\varphi)_\ast= \rho_{t,m}(\varphi)\circ-$. 
When $\underline{J_m(B)}$ is corepresentable in the category of $R$-algebras, we call a \textsf{$m^{th}$ jet algebra of $B$} a $R$-algebra representing the functor $\underline{J_m(B)}$. 
\end{definition}

We introduce a Definition-Lemma that facilitates the computation of the jet \( R \)-algebra for any algebra of finite type (also discussed in \cite[\S 2]{EM}). Our goal is to provide a comprehensive construction of the natural transformations associated with the computed jet algebra representative, while also establishing our notational conventions. The proof follows straightforwardly from the precise formulation of the statement.

\begin{deflem} \cite[\S 2]{EM} \label{defft}
Let $m\in \mathbb{N}$, $n \in \mathbb{N}^*$, and  $B = R[\xn]/ I$ for some ideal $I$ of $R[\xn]$.
\begin{itemize} 
\item When $m=0$, we set $\mathcal{L}_0 (B) := B$. When $m >0$, the $m^{th}$ jet functor of $B$, $\underline{J_m(B)}$, is corepresented by the finite type $R$-algebra:
\[
\mathcal{L}_m (B)= B[\underline{x}_{n,m}^1 ]/ \mathcal{L}_m(I)
\]
where $\mathcal{L}_m(I)$ is the ideal of $B[\underline{x}_{n,m}^1 ]$ defined as:
\[
\mathcal{L}_m(I) = \left< g^m_{h,j}|  h\in I\right>,
\]
where, for all $h\in R[\xn]$, the elements $g_{h,1}^m,\dots, g_{h,m}^m\in B[\underline{x}_{n,m}^1]$ are uniquely defined by the following equation:
\[
h\left( \left(  [x_{i}]_I+ \sum_{j=1}^m  x_{i,j} [t]_{m}^j \right)_{\! i\in \nn{n}} \right) = \sum_{j=0}^m g^m_{h, j} [t]_{m}^j.
\]
In the remainder of the paper, we set $x_{i,0}:= [x_i]_I$, for all $(i,j)\in \nn{n} \times \nn{m}$. 

We will refer to the $R$-algebra $\mathcal{L}_m (B)$ as the \textsf{$m^{th}$ standard jet algebra of $B$ with respect to $I$} and the ideal $\mathcal{L}_m(I)$ as the \textsf{$m^{th}$ standard jet ideal of $I$}. 

\item Two natural correspondences showing the representability of the jet algebra functor via the jet algebra can be defined as follows. First, there is a natural correspondence $\eta_m^B: \underline{J_m(B)}\rightarrow \operatorname{Hom}_{ \mathbf{Alg}_R}( \mathcal{L}_m(B) , - )$ given as the family $\left( {\eta_m^B}_C\right)_{C \in\mathbf{Alg}_R}$ such that:
\begin{align*} {\eta_m^B}_C:    \operatorname{Hom}_{ \mathbf{Alg}_R}( \mathcal{L}_m(B) , C ) & \rightarrow  \operatorname{Hom}_{ \mathbf{Alg}_R}( B , C [t]/ (t^{m+1})) \\
 \psi & \mapsto   {\eta_m^B}_C(\psi),
\end{align*}
where $C$ is an $R$-algebra and ${\eta_m^B}_C(\psi) : B  \rightarrow C [t]/ (t^{m+1})$ sends $[x_i]_{I}$ to \\
$\rho_{t, m}(\psi) (\sum_{j=0}^m  [x_{i,j}]_{\mathcal{L}_m(I)} [t]_m^j)=  \sum_{j=0}^m \psi ( [x_{i,j}]_{\mathcal{L}_m(I)}) [t]_m^j$, for all $i \in \nn{n} $.

On the other hand, there exists a natural correspondence $\epsilon_m^B:  \operatorname{Hom}_{ \mathbf{Alg}_R}( \mathcal{L}_m(B) , - )\rightarrow  \underline{J_m(B)}$ defined as the family $\left( {\epsilon_m^B}_C\right)_{C \in\mathbf{Alg}_R}$ such that
\begin{align*} {\epsilon_m^B}_C:    \operatorname{Hom}_{ \mathbf{Alg}_R}( B , C [t]/ (t^{m+1})) & \rightarrow    \operatorname{Hom}_{ \mathbf{Alg}_R}( \mathcal{L}_m(B) , C ) \\
 \phi & \mapsto    {\epsilon_m^B}_C( \phi) 
\end{align*}
where $C$ is an $R$-algebra and $ {\epsilon_m^B}_C(\phi) : \mathcal{L}_m(B) \rightarrow C $ sends $[x_{i,j}]_{\mathcal{L}_m(I)}$ to ${\epsilon_m^B}_C(\phi)([x_{i,j}]_{\mathcal{L}_m(I)})$ determined uniquely by the equality $\phi ( [x_i]_I) = \sum_{j=0}^m {\epsilon_m^B}_C(\phi)([x_{i,j}]_{\mathcal{L}_m(I)}) [t]_m^j$, for all $i \in \nn{n} $. Moreover, $\eta_m^B$ and $\epsilon_m^B$ are inverses of each other.
\end{itemize}
\end{deflem}


\begin{remark} \label{quotient1}
From the constructions of the jet structure, the following lemma is straightforward to derive the following: Let $n\in \mathbb{N}^*$, $m\in \mathbb{N}$ and $B = R[\xn]/ I$ for some ideal $I$ of $R[\xn]$. For all $C$ $B$-algebras, we have 
\[  \mathcal{L}_m( B) \otimes_{B} C\simeq \frac{C[ \underline{x}_{n,m}^1]}{\pi_C ( \mathcal{L}_m( I))}\]
where $\pi_C $ is the canonical map from $B[  \underline{x}_{n,m}^1]$ to $C[  \underline{x}_{n,m}^1]$ induced by the structural morphism from $B$ to $C$ turning $C$ into a $B$-algebra. 
\end{remark}
We omit the proofs of the next two Definition-Lemmas as they can be easily verified by straightforward computations explicitly suggested in the statements and pertain to well-known properties of jet algebras. The following Definition-Lemma explicitly describes some maps associated with the jet algebra.

\begin{deflem} \cite[\S 2]{EM} \label{defftt}
Let $n \in \mathbb{N}^*$, $m\in \mathbb{N}$, and $s\in \mathbb{N}$ with $s\geq m$ and $B \colon= R[\xn]/ I$ for some ideal $I$ of $R[\xn]$.
\begin{enumerate}  
\item We define the \textsf{$(m,s)$-jet extension of $B$} to be the $R$-algebra morphism
\[
\mu^B_{m,s}  := {\epsilon_m^B}_{\mathcal{L}_s(B)} ( \tau_{m, s}^{\mathcal{L}_s(B)} \circ {\eta_s^B}_{\! \! \mathcal{L}_s(B)} ( \operatorname{id}_{\mathcal{L}_s(B)})).
\]
In other words,
$$\begin{array}{llll}
\mu^B_{m,s} \colon & \mathcal{L}_m (B) & \rightarrow  &\mathcal{L}_s(B) \\
& [x_{i,j}]_{\mathcal{L}_m (I)} &\mapsto  & [x_{i,j}]_{\mathcal{L}_s (I)}  .
\end{array}$$
We note that ${\mu^{B}_{m,s}}^\ast\circ {\epsilon_{s}^B}_C={\epsilon_{m}^B}_C\circ {\tau_{m,s}^C}_{\ast}$ where $C$ is an $R$-algebra.

\item 
We define the \textsf{$m$-jet section to $B$} to be the map  $\sigma_m^B \colon = {\epsilon_{m}^B}_B(\iota^B_m) \colon \mathcal{L}_m(B)\rightarrow B$ where $\iota_m^B:  B \rightarrow B [t]/ (t^{m+1})$ is the canonical inclusion. We note that $\sigma_m^B\circ \mu^B_{0,m}=\operatorname{id}_B$. 

\item We introduce the following canonical $R$-algebra morphism $\upsilon^B_{m,s}:={\eta_m^B}_{\mathcal{L}_s(B)} ( \mu_{m,s}^B) $:
$$\begin{array}{llll}
\upsilon_{m,s}^B:&  B & \rightarrow & \mathcal{L}_s(B)[t]/ (t^{m+1})\\ 
 & [x_i]_{I} & \mapsto &  \sum_{j=0}^{m} [x_{i,j}]_{\mathcal{L}_s(I)} [t]_m^j.
\end{array}$$
\end{enumerate}
\end{deflem}
We will now construct a functorial mapping between morphisms of algebras and morphisms between their jet algebras:
\begin{deflem} \label{dtt} 
Let $m \in \mathbb{N}$, $n_1, n_2 \in \mathbb{N}^*$, $B = R[\underline{x}_{n_1}]/I_1$ for some ideal $I_1$ of $R[\underline{x}_{n_1}]$ and $C = R[\underline{y}_{n_2}]/I_2$ for some ideal $I_2$ of $R[\underline{x}_{n_2}]$. Consider a morphism of $R$-algebras $\phi\colon C\rightarrow B$. We define the morphism of $R$-algebras $\mathcal{L}_m(\phi)\colon \mathcal{L}_m(C) \rightarrow \mathcal{L}_m(B)$ defined by
\[
\mathcal{L}_m(\phi) = \epsilon_{m,\mathcal{L}_m(B)}^C (  \eta^B_{m,\mathcal{L}_m(B)} (\operatorname{id}_{\mathcal{L}_m(B)})\circ \phi).
\]
The following diagram of natural transformations is commutative: 
$$
\xymatrix{ \underline{J_m(C)} \ar@/^1pc/[rr]^-{\epsilon_{m}^C} \ar[dd]_{\phi^\ast}&& \ar@/^1pc/[ll]^-{\eta_{m}^C}  \operatorname{Hom}_{ \mathbf{Alg}_R}( \mathcal{L}_m(C) , - )\ar[dd]^-{\mathcal{L}_m(\phi)^\ast} \\ \\
\underline{J_m(B)} \ar@/^1pc/[rr]^-{\epsilon_{m}^B} && \ar@/^1pc/[ll]^-{\eta_{m}^B} \operatorname{Hom}_{ \mathbf{Alg}_R}( \mathcal{L}_m(B) , - )}.
$$ 
This morphism can be explicitly expressed as follows: for all $i \in \on{n_2}$, let $[p_i]_{I_1} := \phi([y_i]_{I_2})$, so that $p_i((\sum_{j=0}^m [x_{i, j}]_{\mathcal{L}_m (I_1)} [t]_m^j)_{i\in I_1}) = \sum_{j=0}^s [q^m_{i,j}]_{\mathcal{L}_m (I_1)}[t]_m^j$, for some $q_{i,j}^m\in R[\underline{y}_{n_2}]$, with $(i,j)\in \on{n_1} \times \on{n_2}$. Then,
$$\begin{array}{llll}
\mathcal{L}_m(\phi):  & \mathcal{L}_m(C) & \rightarrow & \mathcal{L}_m(B) \\
 &[y_{i,j}]_{\mathcal{L}_m(I_2)} & \mapsto  & [q^m_{i,j}]_{\mathcal{L}_m(I_1)}.
\end{array}$$
Moreover, $  \mathcal{L}_s(\phi)\circ\mu^C_{m,s}  = \mu^B_{m,s} \circ \mathcal{L}_m(\phi) $ where $s\geq m$ and $\mathcal{L}_m(\operatorname{id}_B) = \operatorname{id}_{\mathcal{L}_m(B)}$. Furthermore, for any morphism $\varphi: D \rightarrow C$, we have $\mathcal{L}_m(\phi \circ \varphi) = \mathcal{L}_m(\phi) \circ \mathcal{L}_m(\varphi)$.
\end{deflem}

We quickly revisit the construction of a Jet algebra for general algebra here. By \cite[Lemma 10.127.2]{Stack}, there exists a directed system $B_\lambda$, $\lambda \in \Lambda$, of $R$-algebras of finite presentation together with $R$-algebra morphism $\pi_{\lambda_1, \lambda_2}: B_{\lambda_1} \rightarrow B_{\lambda_2}$ giving the family $( B_\lambda )_{\lambda \in \Lambda}$ its structure of directed system, such that $B = \varinjlim B_{\lambda}$. Then $\underline{J_m(B)}$ is corepresented by $\varinjlim \mathcal{L}_m(B_\lambda)$, where the maps $\mathcal{L}_m(\pi_{\lambda_1, \lambda_2}) \colon \mathcal{L}_m(B_{\lambda_1})\rightarrow \mathcal{L}_m(B_{\lambda_2})$ give the direct system structure. This follows directly from the universal property defining direct limits and the fact that direct limits commute with the $\operatorname{Hom}$ functor.
 Hence, we can define, up to unique isomorphism, a $m$-th jet algebra of an arbitrary $R$-algebra:

\begin{definition}\label{lm}
Let $m \in \mathbb{N}$, $B$ an arbitrary $R$-algebra and let $(B_\lambda)_{\lambda \in \Lambda}$ be a direct system such that $\varinjlim B_\lambda = B$. We define a $m$-th jet algebra of $B$ associated with $\left( B_\lambda\right)_{\lambda \in \Lambda}$  as
\[
 \mathcal{L}_m(B) \colon = \varinjlim \mathcal{L}_m(B_\lambda).
\]
For every $R$-algebra morphism $\phi : B \rightarrow C$, we define a morphism $\mathcal{L}_m(\phi): \mathcal{L}_m (B) \rightarrow \mathcal{L}_m(C)$  as
\[    \mathcal{L}_m(\phi) = \epsilon_{m,\mathcal{L}_m(B)}^C (  \eta^B_{m,\mathcal{L}_m(B)} (\operatorname{id}_{\mathcal{L}_m(B)})\circ \phi)   . 
\] 
It follows from the previous arguments, $\mathcal{L}_m(-)\colon \mathbf{Alg}_R \rightarrow \mathbf{Alg}_R$ is a functor, which we call $m$-th \textsf{jet algebra functor}.
\end{definition}

\begin{remark} \label{locp}
Let $m \in \mathbb{N}$. From the construction in Definition-Lemma \ref{defft} and the preceding discussion, we deduce the following:
\begin{enumerate} 
\item $B$ can be identified with a subalgebra of $\mathcal{L}_m(B)$. Therefore, for any \( b \in B \), its image in $\mathcal{L}_m(B)$ is also denoted by $b$.
\item Let $m \in \mathbb{N}$, $n \in \mathbb{N}^*$, and $B = R[\xn]/ I$ for some ideal $I$ of $R[\xn]$, with $f \in B$. Given that $B_f \simeq R[\xn, t]/ \langle I, ft-1 \rangle$, it follows straightforwardly that $\mathcal{L}_m(B_f) \simeq \mathcal{L}_m(B)_f$.
\item Let $\mathfrak{p}$ be a prime ideal of $B$. Considering $(B_f)_{f \notin \mathfrak{p}}$ as a directed system via the canonical maps $\pi_{f_1, f_2} : B_{f_1} \rightarrow B_{f_2}$ when $f_1 | f_2$, we obtain $B_\mathfrak{p} = \varinjlim B_f$. Combining this with (2) and the construction discussed above, we derive
$$\mathcal{L}_m(B_\mathfrak{p}) \simeq \varinjlim \mathcal{L}_m(B_f) \simeq \varinjlim \mathcal{L}_m(B)_f \simeq \mathcal{L}_m(B)_{\mathfrak{p}\mathcal{L}_m(B)}.$$
\end{enumerate} 
\end{remark}

The following Definition-Lemma presents an alternative perspective on the first jet functor, which is useful to motivate our subsequent discussions.
\begin{deflem}\label{sym}
Let $B$ and $C$ be $R$-algebras. 
\begin{enumerate} 
\item We denote $d_B$ the canonical $R$-derivation from $B$ to $\Omega^1_{B/R}$. 
\item We denote $\varsigma_B$ the canonical structural map from $B$ to $\operatorname{Sym}^\ast ( \Omega^1_{B/R}) $ giving $\operatorname{Sym}^\ast ( \Omega^1_{B/R})$ its structure of $B$-algebra.
\item We denote $\iota_B$ the canonical injection from $\Omega^1_{B/R}$ to $\operatorname{Sym}^\ast ( \Omega^1_{B/R}) $. 
\item We have natural correspondences
\[
\underline{J_1(B)}  \simeq^\Theta  \underline{B} \times \operatorname{Der}_R( B , - )\simeq^\Upsilon \underline{\operatorname{Sym}^\ast ( \Omega^1_{B/R})}  .
\]
More precisely, given an $R$-algebra $C$,
\begin{itemize}
\item For any $\psi \in \underline{J_1(B)}(C)$, $\Theta_C (\psi) = (\psi_0, \psi_1)$, where $\psi_0 \in \operatorname{Hom}_{ \mathbf{Alg}_R} ( B, C) $ and $\psi_1\in \operatorname{Der}_R( B , C)$ are defined by the equality $\psi = \psi_0 + [t] \psi_1$.  
\item For any $\psi_0 \in \operatorname{Hom}_{ \mathbf{Alg}_R} ( B, C) $ and $\psi_1\in \operatorname{Der}_R( B , C)$, $\Theta_C^{-1} (\psi_0, \psi_1)=\psi_0 + [t] \psi_1$.
\item For any $\psi_0 \in \operatorname{Hom}_{ \mathbf{Alg}_R} ( B, C) $ and $\psi_1\in \operatorname{Der}_R( B , C)$, $\Upsilon_C (\psi_0, \psi_1)=\widetilde{\psi_1}$ is defined as follows. By the universal property of $\Omega^1_{B/R}$, considering $C$ as a $B$-module via $\psi_0$, there is the $B$-linear map $h: \Omega^1_{B/R}\rightarrow C$ such that $ \psi_1 = h \circ d_B$ (see \cite[Definition 1.2]{Qing}) and, using the universal property of $ \operatorname{Sym}^\ast ( \Omega^1_{B/R})$ (see \cite[\S 6.1, Proposition 2]{NB}), we obtain the $B$-algebra homomorphism $\widetilde{\psi_1}:  Sym ( \Omega^1_{B/R}) \rightarrow C$ such that $h= \widetilde{\psi_1} \circ \iota_B$.
\item For any $\psi:  \operatorname{Sym}^\ast ( \Omega^1_{B/R}) \rightarrow C$,  $\Upsilon_C^{-1} (\psi)=(\psi_0, \psi_1)$ with $\psi_0= \psi \circ \varsigma_B$ and $\psi_1= \psi \circ \iota_B \circ d_B $;
\end{itemize}

\item The morphism 
$$\chi_B := {\epsilon_1}_{\operatorname{Sym}^\ast ( \Omega^1_{B/R})}^B\circ \Theta^{-1}_{\operatorname{Sym}^\ast ( \Omega^1_{B/R})} \circ \Upsilon^{-1}_{\operatorname{Sym}^\ast ( \Omega^1_{B/R})} (\operatorname{id}_{ \operatorname{Sym}^\ast ( \Omega^1_{B/R})}):\mathcal{L}_1 ( B) \rightarrow \operatorname{Sym}^\ast ( \Omega^1_{B/R})$$ is an $R$-algebra isomorphism from $ \mathcal{L}_1 ( B)$ to $\operatorname{Sym}^\ast ( \Omega^1_{B/R})$ that sends $[x_{i,0}]_{\mathcal{L}_1(I)}$ to $ \varsigma_B ([x_i]_I)$ and $[x_{i,1}]_{\mathcal{L}_1(I)}$ to $\iota_B (d_B( [x_i]_I))$. 

\item We have an isomorphims of $R$-modules
$$ ( \mathcal{L}_1 ( B)\otimes_B {N_s (\mathfrak{p})}) / \langle x_{i,1}\rangle^2_{i\in \nn{n}}( \mathcal{L}_1 ( B)\otimes_B {N_s (\mathfrak{p})}) \simeq  N_s(\mathfrak{p}) \oplus \Big( 
\Omega^1_{B/R}\otimes_B N_s(\mathfrak{p}) \Big),$$
sending $[[x_{i,0}]_{\mathcal{L}_1(I)}\otimes 1]_{\langle x_{i,1}\rangle^2_{i\in \nn{n}}}$ to $\zeta_{\mathfrak{p}, s} (\varsigma_B ([x_i]_I))$ and $[[x_{i,1}]_{\mathcal{L}_1(I)}\otimes 1]_{\langle x_{i,1}\rangle^2_{i\in \nn{n}}}$ to $d_B( [x_i]_I) \otimes 1$. 
\end{enumerate}
\end{deflem}

\begin{proof}
\begin{enumerate}
\item[(4)] Let $\psi \in \underline{J_1(B)}(C)$, and denote $\Theta_C (\psi) = (\psi_0, \psi_1)$, where $\psi_0 \in \operatorname{Hom}_{\mathbf{Alg}_R} ( B, C) $ and $\psi_1 \in \operatorname{Der}_R( B , C)$ are defined by the equality $\psi = \psi_0 + [t] \psi_1$.

Take any $a, b \in B$. As $\psi$ is an $R$-algebra morphism, we have $\psi(ab) = \psi(a) \psi(b)$. Specifically, $\psi(ab) = \psi_0(ab) + [t] \psi_1(ab)$, and $\psi(a)\psi(b) = (\psi_0(a) + [t] \psi_1(a))(\psi_0(b) + [t] \psi_1(b))$. Since $1$ and $[t]$ are linearly independent over $R$, we equate coefficients to obtain $\psi_0(ab) = \psi_0(a)\psi_0(b)$ and $\psi_1(ab) = \psi_0(b)\psi_1(a) + \psi_0(a)\psi_1(b)$. 

Thus, $\psi_0 \in \operatorname{Hom}_{\mathbf{Alg}_R} ( B, C) $ and $\psi_1 \in \operatorname{Der}_R( B , C)$, as $\psi_0$ and $\psi_1$ are $R$-linear. This shows that $\Theta_C$ is well-defined. The subsequent steps of the proof follow straightforwardly from this observation.
\end{enumerate} 
\end{proof}

\section{Thickened fibers of jet algebras}
We begin by recalling that $N_s(\mathfrak{p}) \colon = B_\mathfrak{p} / \mathfrak{p}^s B_\mathfrak{p}$, and $\zeta_{\mathfrak{p}, s}$ denotes the canonical map from $B$ to $N_s(\mathfrak{p})$. The next definition introduces the notion of thickened fibers and localized thickened fibers within the context of a jet algebra functor. These notions are significant for our study of ramification.

\begin{deflem} \label{thickfib}
Let $\mathfrak{p} \subseteq B$ be a prime and $s, m \in \mathbb{N}^\ast$.
\begin{enumerate} 
\item We define the \textsf{$m$-thickened fiber of $\underline{J_s(B)}$}, denoted as $\underline{J_s(B)}_{\mathfrak{p}^m}$, as the subfunctor of $\underline{J_s(B)}$ that associates to any $R$-algebra $C$ the set:
\[
\underline{J_s(B)}_{\mathfrak{p}^m}(C) =\left\{\psi \in \underline{J_s(B)}(C) \colon \tau^C_{0,s} \circ \psi(\mathfrak{p}^m) = 0 \right\}.
\]
We have $\underline{J_s(B)}_{\mathfrak{p}^m}\simeq \underline{J_s(B)} \times_{\underline{B}} \underline{B/\mathfrak{p}^m}$, therefore $\underline{J_s(B)}_{\mathfrak{p}^m}$ is corepresented by $\mathcal{L}_s(B) \otimes_B B/\mathfrak{p}^m$. 

\item We define the \textsf{$m$-localized thickened fiber of $\underline{J_s(B)}$}, denoted as $\underline{J_s(B)}_{N_m(\mathfrak{p})}$, as the subfunctor of $\underline{J_s(B)}$ that associates to any $R$-algebra $C$ the set:
\[
\underline{J_s(B)}_{N_m(\mathfrak{p})}(C) = \left\{\psi \in \underline{J_s(B)}(C) \colon\tau^C_{0,s} \circ \psi(\mathfrak{p}^m) = 0 \text{ and } \tau^C_{0,s} \circ \psi(B)\setminus \mathfrak{p}  \subseteq C^\times \right\}.
\]
We have $\underline{J_s(B)}_{N_m(\mathfrak{p})}\simeq \underline{J_s(B)} \times_{\underline{B}} \underline{N_m(\mathfrak{p})}$, therefore $\underline{J_s(B)}_{N_m(\mathfrak{p})}$ is corepresented by $\mathcal{L}_s(B) \otimes_{B} N_m(\mathfrak{p})$. Moreover, we have the following isomorphisms as $B$-algebras:
$$   \mathcal{L}_s(B) \otimes_{B} N_m(\mathfrak{p})  \simeq \mathcal{L}_s(B_{\mathfrak{p}}) \otimes_{B_{\mathfrak{p}}} N_m(\mathfrak{p})  \simeq  \mathcal{L}_s(B_\mathfrak{p})/(\mathfrak{p}\mathcal{L}_s(B_\mathfrak{p}))^m.$$
\end{enumerate}
\end{deflem}

\begin{proof}
\begin{enumerate}
\item By definition, for any $R$-algebra $C$, we have:
\[
\underline{J_s(B)}\times_{\underline{B}} \underline{B/\mathfrak{p}^m}(C) = \{(\psi, \phi)\in \Hom_R(B, C[t]/(t^{s+1}))\times \Hom_R(B/\mathfrak{p}^s, C)\colon \tau^C_{0,s,\ast}(\psi)= \pi_m^{\ast}(\phi) \}.
\]
where $\pi_m\colon B\rightarrow B/\mathfrak{p}^m$ is the reduction modulo $\mathfrak{p}^m$.

We want to show that the forgetful map
\begin{align*}
\underline{J_s(B)}\times_{\underline{B}} \underline{B/\mathfrak{p}^m}(C) & \rightarrow \underline{J_s(B)}(C) \\
(\psi, \phi) & \mapsto \psi 
\end{align*}
is injective and its image contains $\psi\in \Hom_R(B, C[t]/(t^s))$ if and only if $\psi(\mathfrak{p}^m)\subseteq (t)$.

The forgetful map is injective due to the surjectivity of $\pi_m$. Now let us describe the image of that map; on one hand, for $(\psi, \phi)\in \underline{J_s(B)}\times_{\underline{B}} \underline{B/\mathfrak{p}^m}(C)$, we have that
\[
\tau^C_{0,s,\ast}(\psi)(\mathfrak{p}^m)=\pi_m^\ast(\phi)(\mathfrak{p}^m)=\phi(\pi_m(\mathfrak{p}^m))=\phi(0)=0.
\]
On the other hand, if $\psi\in \underline{J_s(B)}(C)$ has the property that $\psi(\mathfrak{p}^m)\subseteq (t)$, then $\mathfrak{p}^m\subseteq \ker \tau_{0,s,\ast}^C(\psi)$, thus there is a map $\phi\colon B/\mathfrak{p}^m\rightarrow C$ such that $\phi\circ\pi_m=\tau_{0,s,\ast}^C(\psi)$, showing that $\psi$ belongs to the image of the forgetful map.

\item
By definition, for any \( R \)-algebra \( C \), we have
\[
\underline{J_s(B)} \times_{\underline{B}} \underline{N_m(\mathfrak{p})}(C) = \{(\psi, \phi) \in \Hom_R(B, C[t]/(t^{s+1})) \times \Hom_R(N_m(\mathfrak{p}), C) \colon \tau^C_{0,s,\ast}(\psi) = \zeta_{\mathfrak{p},m}^{\ast}(\phi) \}.
\]
The forgetful map
\begin{align*}
\underline{J_s(B)} \times_{\underline{B}} \underline{N_m(\mathfrak{p})}(C) & \rightarrow \underline{J_s(B)}(C) \\
(\psi, \phi) & \mapsto \psi 
\end{align*}
is injective because \( \zeta_{\mathfrak{p},m} \) is an epimorphism, and localization and quotient maps are epimorphisms in the category of commutative unitary rings. According to the universal property of localization and the first isomorphism theorem, the map \( \phi \) exists and is unique if and only if both of the following conditions hold:
\begin{enumerate}
    \item \( \tau^C_{0, s} \circ \psi(\mathfrak{p}^m) = 0 \),
    \item \( \tau^C_{0, s} \circ \psi(B \setminus \mathfrak{p}) \subseteq C^\times \).
\end{enumerate}

The isomorphism \( \mathcal{L}_s(B) \otimes_{B} N_m(\mathfrak{p}) \simeq \mathcal{L}_s(B_{\mathfrak{p}}) \otimes_{B_{\mathfrak{p}}} N_m(\mathfrak{p}) \simeq \mathcal{L}_s(B_\mathfrak{p})/(\mathfrak{p}\mathcal{L}_s(B_\mathfrak{p}))^m \) follows from Remark \ref{locp}.

\end{enumerate}
\end{proof}

There is a natural correspondence between the localized thickened fibers at $\mathfrak{p}\subseteq B$ of its first jet functor 
and the derivations of $N_{m+1}(\mathfrak{p})$.
\begin{lemma} \label{Nsp}
Let $B$ be an $R$-algebra, $\mathfrak{p}$ a prime ideal of $B$, and $m \in \mathbb{N}^\ast$. Given a morphism of rings $\eta\colon N_m(\mathfrak{p}) \rightarrow C$, we define  $\Psi_C$ as follows:

\[
\begin{aligned}
\Psi_C \colon \underline{J_1(B)}_{N_m(\mathfrak{p})}(C) & \rightarrow  \Der_R(N_{m+1}(\mathfrak{p}), C) \\
\phi & \mapsto   \widetilde{\phi_1},
\end{aligned}
\]

\noindent where $\widetilde{\phi_1}\colon N_{m+1}(\mathfrak{p}) \rightarrow C$ is the unique map such that $ \widetilde{\phi_1}\circ \zeta_{\mathfrak{p},m}=\phi_1$. Here, $\phi_1$ is the derivation from $B$ to $C$ defined from the equality $\phi= {\phi}_0 + [t]{\phi}_1$. 

\noindent Then the family of maps $\Psi = (\Psi_C)_{C \in \mathbf{Alg}_{N_m(\mathfrak{p})}}$ forms a natural correspondence from $\underline{J_1(B)}_{N_m(\mathfrak{p})}$ to $\operatorname{Der}_R(N_{m+1}(\mathfrak{p}), -)$. Its inverse is the family $\Phi = (\Phi_C)_{C \in \mathbf{Alg}_{N_m(\mathfrak{p})}}$ such that for any $N_m(\mathfrak{p})$-algebra $C$ with structural morphism $\eta \colon N_m(\mathfrak{p})\rightarrow C$, $\Phi_C$ is defined by:

\[
\begin{aligned}
\Phi_C :  \Der_R(N_{m+1}(\mathfrak{p}), C) &\rightarrow  \underline{J_1(B)}_{N_m(\mathfrak{p})}(C) \\
\psi & \mapsto  \eta \circ \zeta_{\mathfrak{p},m} + [t]\psi \circ \zeta_{\mathfrak{p},m}.
\end{aligned}
\]
\end{lemma}

\begin{proof} 
We begin by proving that \( \Psi_C \) is well-defined. Let \( \psi \in \underline{J_1(B)}_{N_m(\mathfrak{p})}(C) \).
According to Definition-Lemma \ref{thickfib}, \( \psi \in \operatorname{Hom}_R(B, C[t]/(t^{2})) \) and \( \tau_{0,1}^B \circ \psi = \eta \circ \zeta_{\mathfrak{p},m} \). We write \( \psi = \psi_0 + [t] \psi_1 \). Then, \( \psi_0 \in \operatorname{Hom}_R(B , C) \) and \( \psi_1 \in \operatorname{Der}_R(B , C) \) as per Definition-Lemma \ref{sym}.

Let $b \in B \backslash \mathfrak{p}$. We want to prove that $\psi(b) \in C[t]/(t^{2})^\times$. We have $\psi_0(b) \in C^\times$, since $\psi_0 = \eta \circ \zeta_{\mathfrak{p},m}$. Let $c$ be the multiplicative inverse of $\psi_0(b)$. Then $c-[t] \psi_1 (b) c^2$ is the inverse of $\psi(b) $. Indeed, $\psi(b)( c-[t] \psi_1 (b) c^2)  =( \psi_0(b) + [t] \psi_1 (b)) (c-[t] \psi_1 (b) c^2) = 1 +[t] (\psi_1 (b)c - \psi_0(b)\psi_1 (b) c^2 ) = 1$. 
 
By the universal property of localization, there exists a unique map \( \overline{\psi}: B_{\mathfrak{p}} \rightarrow C[t]/(t^2) \) such that \( \psi = \overline{\psi} \circ \ell \) where $\ell : B \rightarrow B_\mathfrak{p}$ is the canonical localization morphism. We write \( \overline{\psi} = \overline{\psi_0} + [t] \overline{\psi_1} \). Then, \( \overline{\psi_0} \in \operatorname{Hom}_R(B_\mathfrak{p} , C) \) and \( \overline{\psi_1} \in \operatorname{Der}_R(B_\mathfrak{p}  , C) \).

Since \( \mathfrak{p}^m B_\mathfrak{p} \subseteq \ker(\overline{\psi_0}) \) and \( \overline{\psi_1} \in \operatorname{Der}_R(B_\mathfrak{p}  , C) \), we have \( \overline{\psi_1}(\mathfrak{p}^{m+1} B_\mathfrak{p}) = 0 \). Thus, \( \mathfrak{p}^{m+1} B_\mathfrak{p} \subseteq \ker(\overline{\psi}) \).

Therefore, by the first isomorphism theorem, \( \overline{\psi} \) induces a morphism of \( R \)-algebras \( \widetilde{\psi}: N_{m+1}(\mathfrak{p}) \rightarrow C[t]/(t^2) \). Thus, there exist \( \widetilde{\psi_0} \in \operatorname{Hom}_R(B_\mathfrak{p} , C) \) and \( \widetilde{\psi_1} \in \operatorname{Der}_R(B_\mathfrak{p}  , C) \) such that \( \widetilde{\psi} = \widetilde{\psi_0} + [t] \widetilde{\psi_1} \). By construction, we have \( \widetilde{\psi_1} \circ \zeta_{\mathfrak{p},m} = \psi_1 \).

This establishes that \( \Psi_C \) is well-defined.

To conclude, it is straightforward to show that \( \Phi_C \) is well-defined, and \( \Psi_C \) and \( \Phi_C \) are inverses of each other.
\end{proof}

The following lemma provides an alternative characterization of the fiber of a jet algebra at a prime ideal. In particular, we establish the irreducibility of the fiber.
\begin{lemma}\label{irreduc}
Let $B$ be a $k$-algebra and let $\mathfrak{p}$ be a prime ideal of $B$ such that $k(\mathfrak{p})$ is separably generated over $k$. Then, there exists an isomorphism of $k(\mathfrak{p})$-algebras:
$$\mathcal{L}_1 (B) \otimes_B  k ( \mathfrak{p})  \simeq F_{d+t} ( k ( \mathfrak{p})),$$
where $d = \dim_{k(\mathfrak{p})}(\mathfrak{p}_\mathfrak{p} / \mathfrak{p}^2_\mathfrak{p})$, $t = \operatorname{tr}(k ( \mathfrak{p}) /k)$, and $F_{d+t} ( k ( \mathfrak{p}))$ is a free $k (\mathfrak{p})$-algebra generated by ${d+t}$ elements. 
In particular, $\mathcal{L}_1 (B) \otimes_B  k ( \mathfrak{p})$ is irreducible, and $\mathfrak{p}\mathcal{L}_1(B)$ is a prime ideal of $\mathcal{L}_1 (B)$.
\end{lemma}

\begin{proof}
Let $C$ be a $k(\mathfrak{p})$-algebra, and let $f : k(\mathfrak{p}) \rightarrow C$ be the structural $k(\mathfrak{p})$-algebra map.

By Lemma \ref{Nsp}, we have an isomorphism of $k(\mathfrak{p})$-modules:
$$ \Hom_{k(\mathfrak{p})}(\mathcal{L}_1(B) \otimes_B k(\mathfrak{p}), C) \simeq \Der_{k}(N_{2}(\mathfrak{p}), C).$$

Since $k(\mathfrak{p})/k$ is separably generated, by Lemma \ref{graded}, we have $N_2(\mathfrak{p}) \simeq k(\mathfrak{p}) \oplus \mathfrak{p}_\mathfrak{p} / \mathfrak{p}^2_\mathfrak{p}$. 

By \cite[Theorem 58]{Matsumura}, the exact sequence:
\begin{equation}\label{seq}
\xymatrix{ 0 \ar[r]& \mathfrak{p}_\mathfrak{p} / \mathfrak{p}^2_\mathfrak{p} \ar[r] & \Omega^1_{N_2(\mathfrak{p})/k} \otimes_{N_2(\mathfrak{p})} k(\mathfrak{p})  \ar[r]& \Omega^1_{k(\mathfrak{p})/k} \ar[r]& 0 }
\end{equation}
splits in the category of $k(\mathfrak{p})$-modules.

Moreover, since we have isomorphism of $k(\mathfrak{p})$-modules $\Omega^1_{k(\mathfrak{p})/k}\simeq k(\mathfrak{p})^t$, and $\mathfrak{p}_\mathfrak{p} / \mathfrak{p}^2_\mathfrak{p}\simeq k(\mathfrak{p})^{d}$, by definition of $d$, we have $\Omega^1_{N_2(\mathfrak{p})/k} \otimes_{N_2(\mathfrak{p})} k(\mathfrak{p})\simeq k(\mathfrak{p})^{d+t}$. 
Thus $\mathcal{L}_1(B) \otimes_B k(\mathfrak{p})\simeq F_{d+t}(k(\mathfrak{p}))$.

By the above, we know that $\mathcal{L}_1(B) \otimes_B k(\mathfrak{p})$ is an integral domain, implying, via the fact that $\mathcal{L}_1(B) \otimes_B k(\mathfrak{p}) \simeq \mathcal{L}_1(B)_{\mathfrak{p}\mathcal{L}_1(B)}/(\mathfrak{p}\mathcal{L}_1(B))_{\mathfrak{p}\mathcal{L}_1(B)}$, that $\mathfrak{p}\mathcal{L}_1(B)$ is a prime ideal of $\mathcal{L}_1(B)$.
\end{proof}

\section{Induced jet action, Taylor map and higher ramification groups}

In this section we study the behavior of the jet algebra functor $\mathcal{L}_m(-)$ when applied to automorphism of algebras. The following lemma proves that if an automorphism is trivial modulo some power of a prime ideal, then the induced automorphism on the corresponding thickened fiber of the jet algebra acts trivially modulo another power of the extension of the prime ideal.

\begin{lemma}\label{gjr}
Let $\mathfrak{p}$ be a fixed prime ideal of $B$, and let $m, s \in \mathbb{N}^\ast$ with $m\leq s$. Consider an automorphism $\sigma \colon B\rightarrow B$ of $R$-algebras such that $\sigma (\mathfrak{p}) =\mathfrak{p}$. If, for all $b \in B$, $\sigma(b)- b \in \mathfrak{p}^{s+1}$, then 
$$\underline{J_m(\sigma )}|_{\underline{J_m(B)}_{\mathfrak{p}^{s-m+1}}} =\operatorname{id}_{\underline{J_m(B)}_{\mathfrak{p}^{s-m+1}} }.$$
\end{lemma}

\begin{proof}
Consider an arbitrary $R$-algebra $C$ and $\psi \in \underline{J_m(B)}_{\mathfrak{p}^{s-m+1}}(C)$. 
We have that 
$$\psi(\mathfrak{p})\subseteq ([t])+\left( \tau_{0,m}^C(\psi(\mathfrak{p})) \right)  ,$$
therefore
$$\psi(\mathfrak{p}^{s+1})\subseteq \sum_{j=0}^{s+1} ([t]^j)\left(\tau_{0,m}^C(\psi(\mathfrak{p}^{s-j+1})) \right) .$$

By definition, $\tau_{0,m}^C(\psi(\mathfrak{p}^{s-m+1}))=0$, therefore $\tau_{0,m}^C(\psi(\mathfrak{p}^{s-j+1}))=0$ for $j\leq m$, while for $j>m$ we have that $t^j=0$.
It follows that $\psi(\mathfrak{p}^{s+1})=0$.

Hence, for all $b\in B$ we have that 
$$ \underline{J_m(\sigma )}(\psi)(b)= \rho_{t,m}(\sigma) ( \psi(b))= \psi \circ \sigma (b)= \psi(b)+ \psi(\sigma(b)-b)=\psi(b)  \, , $$
where the last equality follows from the assumption that for any $b\in B$, $\sigma(b)-b\in \mathfrak{p}^{s+1}$.
\end{proof}

In the following example, we observe that the converse of the above lemma is not generally.

\begin{exmp} \label{ex0}
\begin{enumerate}
\item Let \( k = \mathbb{F}_2(y) \), \( B = \frac{k[x]}{\langle (x^2 - y)^2 \rangle} \), and \( \mathfrak{p} := \frac{\langle x^2 - y \rangle}{\langle (x^2 - y)^2 \rangle} \). Consider the map \( \sigma : B \rightarrow B \) sending \( [x] \) to \( [x + (x^2 - y) x] \). It is an involution, since \( \sigma^2([x]) = \sigma([x + (x^2 - y) x]) = [x + (x^2 - y) x + (x^2 - y) x] = [x] \). We observe that \( \sigma([x]) - [x] \notin \mathfrak{p}^2 \). \\
However, \( \mathcal{L}_1(B) = \frac{k[x_0, x_1]}{\langle (x_0^2 - y_0)^2 \rangle} \), and we have \( \mathcal{L}_1(\sigma)([x_0]) = [x_0] \) and \( \mathcal{L}_1(\sigma)([x_1]) = [x_1] \). In this case, \( k(\mathfrak{p}) \) fails to be separably generated over \( k \). It is well known that separability allows a deeper understanding of higher ramification.

\item Let \( k \) be a field of characteristic \( 2 \), \( B = k[x, y] \), \( \mathfrak{p} = \langle x, y \rangle \), and \( \sigma \) be the automorphism of \( k[x, y] \) sending \( x \) to \( x + y^2 \) and \( y \) to \( y \). Then, \( \sigma(x) - x \notin \mathfrak{p}^3 \). However, \( \mathcal{L}_1(B) = k[x_0, y_0, x_1, y_1] \), \( \mathcal{L}_1(\sigma)(x_0) = x_0 \), \( \mathcal{L}_1(\sigma)(x_1) = x_1 \), and \( \mathcal{L}_1(\sigma)(y_i) = y_i \) for \( i \in \{1, 2\} \). In this situation, \( s = \operatorname{char}(k) \). We observe that the trace map here sends everything to \( 0 \), making it further away from being surjective.
\end{enumerate}
\end{exmp}

We can define a morphism of $R$-algebras that behaves similarly to the well-known Taylor map over algebras of finite type. What sets this map apart is its property of being a map of $R$-algebras. This characteristic proves to be particularly advantageous when dealing with matters related to algebraic geometry.

\begin{deflem}\label{Taylor}
Given $B$ an $R$-algebra, $\mathfrak{p}$ be a prime ideal of $B$, and an integer $s \geq 1$. We define the \textsf{Taylor morphism at $\mathfrak{p}$ of order $s$}, denoted by ${T}_{\mathfrak{p},s}$, as the morphism $${\eta_s}^{B_{\mathfrak{p}}}_{\mathcal{L}_1 (B_{\mathfrak{p}})} (\iota_{\mathfrak{p}, s}) \colon B_{\mathfrak{p}} \rightarrow (\mathcal{L}_1 (B)\otimes_{B} {N_s(\mathfrak{p})})[t]/(t^2),$$ 
where $\iota_{\mathfrak{p}, s}: \mathcal{L}_1 (B_{\mathfrak{p}}) \rightarrow \mathcal{L}_1 (B_{\mathfrak{p}})\otimes_{B_{\mathfrak{p}}} {N_s(\mathfrak{p})} \simeq  \mathcal{L}_1 (B)\otimes_{B} {N_s(\mathfrak{p})} $ 
is the canonical map sending $\frac{b}{1}$ to $b \otimes 1$.
Specifically, when $B = k [\xn]/ \mathfrak{a}$ where $n \in \mathbb{N}^*$ and $\mathfrak{a}$ is an ideal of $k[\xn]$, writing $\mathfrak{p}= \mathfrak{P}/ \mathfrak{a}$ and identifying $B_\mathfrak{p}$ with $k [\xn]_\mathfrak{P}/ \mathfrak{a}_\mathfrak{P}$, the Taylor morphism ${T}_{\mathfrak{p},s}$ sends $\left[\frac{f}{1} \right]_{\mathfrak{a}_\mathfrak{P}}$ to $\left[\displaystyle \left[f\right]_{\mathfrak{a}} + [t] \sum_{j=1}^n \left[ \frac{\partial }{\partial x_{j} }\left(f\right)\right]_{\mathfrak{a}} x_{j,1}  \right]_{ \mathcal{L}_1( {\mathfrak{a}})}\! \!\!\!\!\! \otimes 1$.
\end{deflem}

\begin{proof}
The morphism $\iota_{\mathfrak{p}, s}$ is well-defined by Definition-Lemma \ref{thickfib}.
We consider the map
$$\begin{array}{lllll}
	\widetilde{T}\colon & B_ {\mathfrak{p}}&  \rightarrow &(\mathcal{L}_1 (B)_{\mathfrak{p}\mathcal{L}_1 (B)} \otimes_{B_{\mathfrak{p}}} {N_s(\mathfrak{p})})[t]/(t^2) \\
	& \left[\frac{f}{1}\right]_{\mathfrak{a}_\mathfrak{P}} & \mapsto  &\left[\displaystyle \left[f\right]_{\mathfrak{a}} + [t] \sum_{j=1}^n \left[ \frac{\partial }{\partial x_{j} }\left(f\right)\right]_{\mathfrak{a}} x_{j,1}  \right]_{ \mathcal{L}_1( {\mathfrak{a}})}\! \!\!\!\!\! \otimes 1
\end{array}$$
is a map of $R$-algebras, indeed $\tilde{T}$ is by definition $R$-linear and:
\begin{align*}
\tilde{T}\left( \left[\frac{f}{1}\right]_{\mathfrak{a}_\mathfrak{p}}\right)\cdot \tilde{T}\left( \left[\frac{g}{1}\right]_{\mathfrak{a}_\mathfrak{p}}\right)  
=  & \tilde{T}\left( \left[\frac{fg}{1}\right]_{\mathfrak{a}_\mathfrak{p}}\right).
\end{align*}
Hence, $T_{\mathfrak{p}, s}$ and $\tilde{T}$ coincide, since
$$ T_{\mathfrak{p}, s}\left( \left[\frac{x_i}{1}\right]_{\mathfrak{a}_\mathfrak{P}}\right) = {\eta_s^B}_{\mathcal{L}_1 (B)} (\iota_{\mathfrak{p}, s})\left( \left[\frac{x_i}{1}\right]_{\mathfrak{a}_\mathfrak{P}}\right)=\left[ \left[{x_i}\right]_{\mathfrak{a}} + [t]x_{i,1} \right]_{ \mathcal{L}_1( \mathfrak{a})}   \otimes 1, \quad \forall i \in \nn{n}. $$
\end{proof}


\begin{definition} \label{acttaylor}
Consider an automorphism \( \sigma: B \rightarrow B \) of the \( R \)-algebra \( B \) such that for all \( b \in B \), \( \sigma(\mathfrak{p}) = \mathfrak{p} \) and \( s \in \mathbb{N} \). This induces an automorphism on \( \operatorname{Im}({T}_{\mathfrak{p},s}) \) the image of ${T}_{\mathfrak{p},s}$, denoted \( \tau_{\mathfrak{p},s}(\sigma) \), uniquely defined by
\[ \tau_{\mathfrak{p},s}(\sigma)\left({T}_{\mathfrak{p},s}\left(b\right)\right) = {T}_{\mathfrak{p},s}\left(\sigma\left(b\right)\right), \quad \text{for all } b \in B_\mathfrak{p}. \]
\end{definition}

\begin{remark}
We observe that \( \tau_{\mathfrak{p},s}(\sigma) \circ {T}_{\mathfrak{p},s} = \eta_{m,\mathcal{L}_m(B)}^B(\iota_{\mathfrak{p}, s} \circ \mathcal{L}_1(\sigma)) \). In particular, \( \tau_{\mathfrak{p},s}(\operatorname{id}_B) = \operatorname{id}_{\operatorname{Im}({T}_{\mathfrak{p},s})} \).
\end{remark}

We proceed by proving one of the main results of this paper, which asserts that the kernel of the Taylor morphism of order $s$ at a prime is the prime raised to the power of $s+1$, when the residue field of the involved prime is separably generated .
\begin{introthm1}{Theorem} \label{mt}
Let $k$ be a field, $s\geq 1$, $B$ be a finitely presented $k$-algebra, and $\mathfrak{p}$ be a prime ideal of $B$. 
We assume that $k(\mathfrak{p})$ is separably generated over $k$, and $\operatorname{char}(k)$ is either $0$ or greater than $s$. Then $\operatorname{ker}(T_{\mathfrak{p}, s})=\mathfrak{p}_\mathfrak{p}^{s+1}$.
\end{introthm1} 

\begin{proof} 
Since $B$ is a finitely presented $k$-algebra, we have $B\simeq  \frac{k[\xn]}{\mathfrak{a}}$ for some $n\in \mathbb{N}$ and $\mathfrak{a}$ an ideal of $k[\xn]$. Therefore, without loss of generality, we can assume that $B = \frac{k[\xn]}{\mathfrak{a}}$ and $\mathfrak{p}=\mathfrak{P}/\mathfrak{a}$, where $\mathfrak{P}$ is a prime ideal of $k[x_1, \dots , x_n]$ containing $\mathfrak{a}$ by the prime ideal correspondence in quotients. Thus, from Lemma \ref{ss-1} and the isomorphism $ \mathcal{L}_s(B) \otimes_{B} N_m(\mathfrak{p})  \simeq  \mathcal{L}_s(B_\mathfrak{p})/(\mathfrak{p}\mathcal{L}_s(B_\mathfrak{p}))^m$ in Definition-Lemma \ref{thickfib}, we obtain:
\begin{align*} 
\operatorname{ker}( T_{\mathfrak{p},s}) &= \left\{ \left[f\right]_{\mathfrak{a}_\mathfrak{p}} \in B_\mathfrak{p} |\ \left[f\right]_{\mathfrak{a}_\mathfrak{p}} \in\mathfrak{p}_\mathfrak{p}^{s} \text{ and for } i\in \nn{n}, \left[\left(\frac{\partial}{\partial x_i} \right)_{\mathfrak{P}} \left(f \right) \right]_{\mathfrak{a}_\mathfrak{p}} \in \mathfrak{p}_\mathfrak{p}^{s} \right\} \\
&= \left\{   \left[f\right]_{\mathfrak{a}_\mathfrak{p}} \in B_\mathfrak{p} | \ f \in\mathfrak{P}^{s}_\mathfrak{P} \text{ and for } i \in \nn{n}, \left(\frac{\partial}{\partial x_i} \right)_{\mathfrak{P}} \left( f \right) \in \mathfrak{P}^{s}_\mathfrak{P} \right\} = \mathfrak{p}^{s+1}_\mathfrak{p} 
\end{align*}
Therefore, the result follows.
\end{proof} 

We can now establish the equivalence between the behavior of an automorphism on a fat point and the corresponding automorphism on the corresponding thickened jet-fiber.

\begin{cor}\label{Lram}
Let $k$ be a field, $s\geq 1$, $B$ be a finitely presented $k$-algebra, and $\mathfrak{p}$ a prime ideal of $B$ such that $k(\mathfrak{p})$ is separably generated over $k$, and $\operatorname{char}(k)$ is either $0$ or greater than $s$. Consider an automorphism $\sigma: B \rightarrow B$ of the $k$-algebra $B$ such that $\sigma(\mathfrak{p}) = \mathfrak{p}$.

The following statements are equivalent:
\begin{enumerate}
\item For all $b \in B$, $\sigma(b) - b \in \mathfrak{p}^{s+1}$,
\item For all $b \in B_\mathfrak{p}$, $\sigma(b) - b \in \mathfrak{p}^{s+1}B_\mathfrak{p}$,
\item $\underline{J_1(\sigma)}|_{\underline{J_1(B)}_{N_s(\mathfrak{p})}} = \operatorname{id}_{\underline{J_1(B)}_{N_s(\mathfrak{p})}}$,
\item $\overline{\mathcal{L}_1(\sigma) \otimes \operatorname{id}_{N_s(\mathfrak{p})}} = \operatorname{id}_{(\mathcal{L}_1(B)\otimes_B N_s(\mathfrak{p}) )/ \langle x_{i,1}\rangle^2_{i\in \nn{n}}}$, where $\overline{\mathcal{L}_1(\sigma) \otimes \operatorname{id}_{N_s(\mathfrak{p})}}$ is the automorphism induced by $\mathcal{L}_1(\sigma) \otimes \operatorname{id}_{N_s(\mathfrak{p})}$ in $(\mathcal{L}_1(B)\otimes_B N_s(\mathfrak{p}))/ \langle x_{i,1}\rangle^2_{i\in \nn{n}}$. 
\end{enumerate}
\end{cor}

\begin{proof}
$(1) \Rightarrow (2)$ Suppose $(1)$ is true. Then, we have, for any $b \in B$ and $a \in B \setminus \mathfrak{p}$, we have 
$\sigma (b) -b \in \mathfrak{p}^{s+1}$ and $\sigma \left( \frac{1}{a}\right) -\frac{1}{a} \in \mathfrak{p}^{s+1}$. From this, we can easily deduce 
$$\sigma \left( \frac{b}{a}\right) -\frac{b}{a} = \frac{\sigma (b) a - b\sigma(a)}{a \sigma (a)} = \frac{b( a -\sigma(a))}{a \sigma (a)} \in \mathfrak{p}^{s+1}.$$

$(2) \Rightarrow (3)$ follows from Lemma \ref{gjr}.

$(3) \Rightarrow (4)$ is clear.

$(4) \Rightarrow (1)$:We prove this implication by contrapositive. 
Assume that there exists $b \in B$ such that $\sigma(b) - b \notin \mathfrak{p}^{s+1}$. For such a $b$, we have, by Definition \ref{acttaylor} and Theorem \ref{mt}, that $\tau_{\mathfrak{p},s}(\sigma)\left(T_{\mathfrak{p}, s}\left(\frac{b}{1}\right)\right) = T_{\mathfrak{p}, s}\left(\frac{\sigma(b)}{1}\right) \neq T_{\mathfrak{p}, s}\left(\frac{b}{1}\right)$. This proves that $\mathcal{L}_1(\sigma) \otimes \operatorname{id}_{ N_s(\mathfrak{p}) } \neq \operatorname{id}_{\mathcal{L}_1(B)\otimes_B N_s(\mathfrak{p})}$, since $\tau_{\mathfrak{p},s}(\sigma) \circ T_{\mathfrak{p},s} = \eta_{m,\mathcal{L}_m(B)}^B(\iota_{\mathfrak{p},s} \circ \mathcal{L}_1(\sigma))$.
Finally, 
$$\overline{\mathcal{L}_1(\sigma) \otimes \operatorname{id}_{N_s(\mathfrak{p})}} \neq  \operatorname{id}_{(\mathcal{L}_1(B)\otimes_B N_s(\mathfrak{p}) )/ \langle x_{i,1}\rangle^2_{i\in \nn{n}}},$$ since 
$$\operatorname{Im}(T_{\mathfrak{p},s}) \subseteq (B_{\mathfrak{p}} x_{1,1}\oplus B_{\mathfrak{p}} x_{2,1}\oplus \dots  \oplus B_{\mathfrak{p}} x_{n,1})\otimes_B N_s (\mathfrak{p}).$$ 
\end{proof}

Given an action of a group on an $R$-algebra, the jet algebra functors induce actions of the same group on the corresponding jet algebras:
\begin{definition} Let $m\in \mathbb{N}$, $G$ be an abstract group, and let $B$ be a $R$-algebra.
Consider an action $\alpha  \colon G \times B \rightarrow B$, so that for $g\in G$, we have an automorphism $\alpha_g\colon B \rightarrow B$ given by $\alpha_g(b)= \alpha(g, b)$, for any $b\in B$.
The group action $\mathcal{L}_m(\alpha) \colon G \times \mathcal{L}_m( B) \rightarrow \mathcal{L}_m (B)$ sending $(g, a)$ to $\mathcal{L}_m ( \alpha_g)(a)$ is called the {\sf  action induced by $\alpha$ on the $m$-th jet algebra of $B$}.
\end{definition} 

\begin{definition}\label{ramification}
We consider an action $\alpha  \colon G \times B \rightarrow B$ and an integer $s\geq -1$, we denote by $G_s(\mathfrak{q})$ the  ramification group of order $s$ from the action $\mathcal{L}_m(\alpha)$ at the prime ideal $\mathfrak{q}$ of $\mathcal{L}_m(B)$. That is the set of elements of $G$ that induce a trivial action on $N_s (\mathfrak{q})$.
\end{definition}


We conclude the paper with the following theorem. This compares the ramification groups of a certain action on an $R$-algebra with the ramification groups of lower order for the induced action on the first jet algebra. In the end, by iterating this construction, it will be possible to express all the ramifications groups of said action in terms of the inertia groups of the actions induced on a suitable algebra. This construction is functorial, meaning that it can be extended to actions of algebraic groups on group schemes, allowing to define an analog of the ramification groups also in that category.
\begin{introthm1}{Theorem}\label{hiram}
Let $B$ be a finitely presented algebra over $k$, $\mathfrak{p}$ be a prime ideal of $B$, $s\in \mathbb{N}$. We assume that $k(\mathfrak{p})$ be separably generated over $k$, and $\operatorname{char}(k)$ is either $0$ or greater than $s$. Then, $\mathfrak{p}\mathcal{L}_1(B)$ is a prime ideal of $\mathcal{L}_1(B)$, and
$$G_{s}(\mathfrak{p}) = G_{s-1} (\mathfrak{p}\mathcal{L}_1(B)) = G_{0} (\mathfrak{p}\mathcal{L}_1^{s}(B))$$
where $\mathcal{L}_1^{t}$ being $\mathcal{L}_1$ composed with itself $t$ times, where $t\in \mathbb{N} \setminus \{ 0 \}$.
\end{introthm1}

\begin{proof} 
By Lemma \ref{irreduc}, we can establish that $\mathfrak{p}\mathcal{L}_1(B)$ is indeed a prime ideal. Moreover, Corollary \ref{Lram} implies that $G_{s}(\mathfrak{p}) = G_{s-1} ( \mathfrak{p}\mathcal{L}_1(B)) .$
\end{proof}

\begin{remark} From Corollary \ref{specmsep}, we obtain 
$$G_s (\mathfrak{p}\mathcal{L}_1(B) ) =\cap_{\mathfrak{n} \in  \operatorname{Specm}_{\operatorname{sep}}(\mathcal{L}_1(B)) \cap V  (\mathfrak{p} \mathcal{L}_1(B))} G_s (\mathfrak{n})=\cap_{\mathfrak{m} \in  \operatorname{Specm}_{\operatorname{sep}}(B) \cap V ( \mathfrak{p})} G_s (\mathfrak{m} \mathcal{L}_1(B)) .$$
\end{remark}
We conclude the paper with several examples to illustrate our theorem.
\begin{exmps} \label{ex1}
\begin{enumerate} 
\item We consider the polynomial ring $k[x,y]$ with an action $\sigma$ of $S_2$ on the indeterminate whose non-trivial element exchanges the variables. We have $G_0(x,y) = S_2$ and $G_1(x,y)= \{ \operatorname{id}\}$. The first jet algebra of $k[x,y]$ is $k[x_0, y_0, x_1, y_1]$ and $\mathcal{L}_1( \sigma )(x_0) = y_0$ and $\mathcal{L}_1( \sigma )(x_1)=y_1$, so that $G_0(x_0,y_0) =G_1(x,y)= \{ \operatorname{id}\}.$
\item We consider a field of characteristic $3$, $G= \langle \sigma \rangle$ which is a group of order $3$, $B=k[x,y, z]/ ( x^3 +y^3 + z^2),$ with  
 $ \sigma ( x) = x+ z^2 , \ \sigma (y) = y- z^2, \ \sigma ( z) =z.$
We have  $$G = G_0(x, y,z)= G_1 (x,y,z)\text{ and } G_{2}(x,y, z)= \{ \operatorname{id}\}.$$
 The first jet algebra is:  
$$\mathcal{L}_1( B) = \frac{k[x_0, y_0,z_0, x_1, y_1, z_1]}{( x_0^3 + y_0^3 + z_0^2, 2 z_0 z_1)}.$$  
The action becomes
 $$\mathcal{L}_1( \sigma )( x_0) = x_0+ z_0^2 , \ \mathcal{L}_1( \sigma ) (y_0) = y_0- z_0^2,$$
  $$\mathcal{L}_1( \sigma ) ( x_1) = x_1+ 2z_0z_1 , \ \mathcal{L}_1( \sigma ) (y_1) = y_1- 2z_0z_1, \ \mathcal{L}_1( \sigma ) ( z_i) =z_i, \ i \in \{ 1, 2\}.$$ 
So that, 
  $$G=G_0( x_0,y_0,z_0) = G_1 (x,y,z) \text{ and } G_1 (x_0,y_0,z_0)= G_2 ( x, y,z)= \{ \operatorname{id}\}. $$
The first jet algebra of the first jet algebra is:
$$\scalebox{1}{$\mathcal{L}_1^2
( B) = \frac{k[(x_{i,j}, y_{i,j} , z_{i,j})_{(i,j)\in \on{1}\times \on{1}}]}{( x_{0,0}^3 + y_{0,0}^3 + z_{0,0}^2, 2 z_{0,0} z_{1,0}, 2 z_{0,0} z_{0,1} , 2 ( z_{0,1}z_{1,0} +z_{0,0} z_{1,1}))} .$}$$  
So that,
 $$ \mathcal{L}_1^2( \sigma ) ( x_{0,0}) = x_{0,0}+ z_{0,0}^2 , \ \mathcal{L}_1^2( \sigma )  (y_{0,0}) = y_{0,0}- z_{0,0}^2,$$ 
  $$ \mathcal{L}_1^2( \sigma )  ( x_{i,j}) = x_{i,j}+ 2z_{0,0}z_{i,j} , \ \mathcal{L}_1^2( \sigma )  (y_{i,j}) = y_{i,j}- 2z_{0,0}z_{i,j}, \text{ for any } (i,j) \in \{ (1,0), (0,1)\},$$ 
      $$ \mathcal{L}_1^2( \sigma ) ( x_{1,1}) = x_{1,1}+ 2(z_{1,0}z_{0,1}+ z_{0,0} z_{1,1}) , \ \mathcal{L}_1^2( \sigma )  (y_{1,1}) = y_{1,1}- 2(z_{1,0}z_{0,1}+ z_{0,0} z_{1,1}),$$ 
      $$ \mathcal{L}_1^2( \sigma )  ( z_{i,j}) =z_{i,j}, \text{ for any } (i,j) \in \{ (0,0), (1,0), (0,1), (1,1)\} .$$ 
and $G_0( x_{0,0},y_{0,0},z_{0,0}) = G_2 (x,y,z)= \{ \operatorname{id}\}.$

        \end{enumerate}
\end{exmps}

\end{document}